\documentclass[<opyions>]{elsarticle}
\usepackage{amsfonts}
\usepackage{tipa}
\usepackage{mathrsfs}
\usepackage{amssymb}
\usepackage{latexsym,bm}
\usepackage{amsthm}
\usepackage{yhmath}
\usepackage{booktabs}
\usepackage{multirow}
\usepackage[center]{caption}
\usepackage{graphics}
\usepackage{subfig}
\usepackage{color}
\usepackage{geometry}

\newtheorem{lemma}{Lemma}

\newtheorem{thm}{Theorem}
\newtheorem{Rem}{Remark}

\begin{document}
\begin{frontmatter}
\title{Novel numerical analysis of multi-term time fractional viscoelastic non-Newtonian fluid models for simulating unsteady MHD Couette flow of a generalized Oldroyd-B fluid}
\author[els]{Libo~Feng}
\ead{fenglibo2012@126.com}
\author[els]{Fawang Liu\corref{cor1}}
\ead{f.liu@qut.edu.au}\cortext[cor1]{Corresponding author.}
\author[els,qv]{Ian Turner}
\ead{i.turner@qut.edu.au}
\author[rvt]{Liancun Zheng}
\ead{liancunzheng@ustb.edu.cn}
\address[els]{School of Mathematical Sciences, Queensland University of
Technology,  GPO Box 2434, Brisbane, Qld 4001, Australia}
\address[qv]{Australian Research Council Centre of Excellence for Mathematical and Statistical Frontiers (ACEMS), Queensland University of Technology (QUT), Brisbane, Australia}
\address[rvt]{School of Mathematics and Physics, University of Science and Technology Beijing, Beijing 100083, China}
\begin{abstract}
In recent years, non-Newtonian fluids have received much attention due to their numerous applications, such as plastic manufacture and extrusion of polymer fluids. They are more complex than Newtonian fluids because the relationship between shear stress and shear rate is nonlinear. One particular subclass of non-Newtonian fluids is the generalized Oldroyd-B fluid, which is modelled using terms involving multi-term time fractional diffusion and reaction. In this paper, we consider the application of the finite difference method for this class of novel multi-term time fractional viscoelastic non-Newtonian fluid models. An important contribution of the work is that the new model not only has a multi-term time derivative, of which the fractional order indices range from 0 to 2, but also possesses a special time fractional operator on the spatial derivative that is challenging to approximate. There appears to be no literature reported on the numerical solution of this type of equation. We derive two new different finite difference schemes to approximate the model. Then we establish the stability and convergence analysis of these schemes based on the discrete $H^1$ norm and prove that their accuracy is of $O(\tau+h^2)$ and $O(\tau^{\min\{3-\gamma_s,2-\alpha_q,2-\beta\}}+h^2)$, respectively. Finally, we verify our methods using two numerical examples and apply the schemes to simulate an unsteady magnetohydrodynamic (MHD) Couette flow of a generalized Oldroyd-B fluid model. Our methods are effective and can be extended to solve other non-Newtonian fluid models such as the generalized Maxwell fluid model, the generalized second grade fluid model and the generalized Burgers fluid model.
\end{abstract}

\end{frontmatter}
\section{Introduction}

Generally, a constitutive equation is used to specify the rheological properties of a material, which is a relation between the stress and the local properties of the fluid. Some common fluids, such as water, oil, air, ethanol and benzene, exhibit a linear relationship between the stress tensor and the rate of deformation tensor, which are called Newtonian fluids.  The Newtonian constitutive equation is the simplest linear viscoelastic model. For small deformations, low stress, low rate, and linear materials, linear viscoelasticity is usually applicable. However, some fluids produced industrially do not obey the Newtonian postulate, such as molten plastics, slurries, emulsions, pulps, and these are termed as non-Newtonian fluids. This means that the rapport between the stress tensor and the rate of deformation tensor is not linear but is non-linear. In reality about 90\% of fluids are nonlinear with large deformations, therefore nonlinear viscoelastic mathematical models are needed. Research related to non-Newtonian fluid mechanics is of great realistic significance to industry. Since a rheometer can not provide the necessary information of important rheological properties, the constitutive equations are the best available tools for understanding the complex behaviour of a material. Due to the nonlinear relationship between stress and deformation and there being no standard form universally valid for each non-Newtonian fluid, the constitutive equation of non-Newtonian fluids is much more complex than its Newtonian counterpart. The constitutive equations involving fractional calculus have proved to be a valuable tool for handling viscoelastic properties \cite{Bagley1,Friedrich1} and some results are obtained that are in good agreement with experimental data \cite{Makris,Hern}.

One particular subclass of non-Newtonian fluids is the generalized Oldroyd-B fluid, which has been found to approximate the response of many dilute polymeric liquids. Consider the flow of an incompressible Olyroyd-B fluid bounded by two infinite parallel rigid plates. Initially, the whole system is at rest and the upper plate is fixed. Then at time $t=0^+$, the lower plate starts to move with some acceleration. Due to the shear effects, the fluid over the plate is gradually disturbed.
The fundamental equations of an incompressible fluid are
\begin{align*}
\text{div}\, \mathbf{V}=0,\quad \rho \frac{\text{d}\mathbf{V}}{\text{dt}}=\text{div}\,\mathbf{T},
\end{align*}
where $\text{div}$ is the divergence operator, $\rho$ is the density of the fluid, $\mathbf{T}$ is the Cauchy stress tensor and $\frac{\text{d}}{\text{d}t}$ is the material time derivative. The constitutive equation for a generalized Oldroyd-B fluid is defined as \cite{Hilfer}:
\begin{align*}
\mathbf{T}=-p\mathbf{I}+\mathbf{S}, \quad
\bigg( 1+\lambda \frac{\text{D}^\alpha}{\text{Dt}^\alpha}\bigg)\mathbf{S}=\mu\bigg(1+\theta \frac{\text{D}^\beta}{\text{Dt}^\beta} \bigg)\mathbf{A},
\end{align*}
where $p$ is the pressure, $\mathbf{I}$ is the identity tensor, $\mathbf{S}$ is the extra-stress tensor, $\lambda$ is the relaxation time, $\mu$ is the dynamic viscosity coefficient of the fluid, $\theta$ is the retardation time, and $\mathbf{A}=\mathbf{L}+\mathbf{L}^T$ ($\mathbf{L}=\nabla \mathbf{V}$) denotes the first Rivlin-Ericksen tensor. The operators $\frac{\text{D}^\alpha}{\text{Dt}^\alpha}$ and $\frac{\text{D}^\beta}{\text{Dt}^\beta}$
are material derivatives and can be expressed as
\begin{align*}
\frac{\text{D}^\alpha \mathbf{S}}{\text{Dt}^\alpha}&=D_t^\alpha \mathbf{S}+(\mathbf{V}\cdot \nabla)\mathbf{S}-\mathbf{L}\mathbf{S}-\mathbf{S}\mathbf{L}^T,\\
\frac{\text{D}^\beta \mathbf{A}}{\text{Dt}^\beta}&=D_t^\beta \mathbf{A}+(\mathbf{V}\cdot \nabla)\mathbf{A}-\mathbf{L}\mathbf{A}-\mathbf{A}\mathbf{L}^T,
\end{align*}
where $D_t^\alpha$ and $D_t^\beta$ are the time fractional derivative operators of order $\alpha$ and $\beta$, respectively.
Assume that the velocity field and stress has the form
\begin{align*}
\textbf{V}=u(y,t) \textbf{i}, \quad \mathbf{S}=\mathbf{S}(y,t).
\end{align*}
Taking into account the initial condition $\mathbf{S}(y,0)=0$ and in absence of the pressure gradient, one can obtain the following equation with fractional derivative of the velocity of the main flow \cite{Qi07,Khan09}
\begin{align}\label{e01}
(1+\lambda D_t^\alpha)\frac{\partial u(y,t)}{\partial t}=\nu(1+\theta D_t^\beta)\frac{\partial^2u(y,t)}{\partial y^2},
\end{align}
where $\nu=\frac{\mu}{\rho}$. When a magnetic field is imposed on the above flow under the assumption of low magnetic Reynolds number, the following velocity equation can be derived \cite{a12,Zheng12}
\begin{align}\label{e02}
(1+\lambda D_t^\alpha)\frac{\partial u(y,t)}{\partial t}=\nu(1+\theta D_t^\beta)\frac{\partial^2u(y,t)}{\partial y^2}-K(1+\lambda D_t^\alpha)u(y,t),
\end{align}
where $K=\frac{\sigma B_0^2}{\rho}$, $B_0$ is the magnetic intensity and $\sigma$ is the electrical conductivity. When the fluid medium is porous, the following magnetohydrodynamic (MHD) flow of a generalized Oldroyd-B fluid with an effect of Hall current can be obtained \cite{Khan06}
\begin{align}\label{e03}
(1+\lambda D_t^\alpha)\frac{\partial u(y,t)}{\partial t}=\nu(1+\theta D_t^\beta)\frac{\partial^2u(y,t)}{\partial y^2}-
\frac{\nu\varphi_1}{k}(1+\theta D_t^\beta)u(y,t)-\frac{\sigma B_0^2}{\rho(1-i\phi)}(1+\lambda D_t^\alpha)u(y,t),
\end{align}
where $k$ is the permeability of the porous medium, $\varphi_1$ is the porosity of the medium, and $\phi$ is the Hall parameter. As  Eqs.(\ref{e01})-(\ref{e03}) contain similar terms, they can be expressed in a generalised form.

In this paper, we will consider the following novel multi-term time fractional non-Newtonian diffusion equation:
\begin{align}
&\sum_{j=1}^{s}b_j\,{D}_t^{\gamma_j}u(x,t)+a_1\frac{\partial u(x,t)}{\partial {t}}+\sum_{l=1}^{q}c_l\,{D}_t^{\alpha_l}u(x,t)+a_2u(x,t)\nonumber\\\label{e04}
=&~a_3\frac{\partial^2 u(x,t)}{\partial {x^2}}+a_4{D_t^\beta}\frac{\partial^2 u(x,t)}{\partial {x^2}}+f(x,t),~(x,t)\in\Omega,
\end{align}
subject to the initial conditions
\begin{equation}\label{e05}
u(x,0)=\phi_1(x),\quad u_t(x,0)=\phi_2(x),\quad 0\le x\le L,
\end{equation}
and the boundary conditions:
\begin{equation}\label{e06}
u(0,t)=0,\quad u(L,t)=0,\quad 0\le t\le T,
\end{equation}
where $a_i>0,~i=1,2,3,4$, $b_j\geq0,~j=1,2,\ldots,s$, $c_l\geq0,~l=1,2,\ldots,q$, $1<\gamma_1<\gamma_2<\ldots<\gamma_s<2$, $0<\alpha_1<\alpha_2<\ldots<\alpha_q<1$ and $\Omega=(0,L)\times(0,T]$.  The Caputo time fractional derivative ${D_t^\beta} u(x,t)$ ($0<\beta<1$) and ${D_t^\gamma} u(x,t)$ ($1<\gamma<2$) are given by \cite{Pod99,Liu15b}
\begin{align*}
{D_t^\beta} u(x,t)&=\frac{1}{\Gamma (1-\beta)}\int_{0}^{t}  (t-s)^{-\beta} \frac{\partial
{u(x,s)}}{\partial {s}}d s,~0<\beta<1,\\
{D_t^\gamma} u(x,t)&=\frac{1}{\Gamma (2-\gamma)}\int_{0}^{t}  (t-s)^{1-\gamma} \frac{\partial^2
{u(x,s)}}{\partial {s^2}}d s,~1<\gamma<2.
\end{align*}
The general multi-term time fractional diffusion equation only contains the multi-term time fractional derivative terms without the special term ${D_t^\beta}\frac{\partial^2 u}{\partial {x^2}}$. Its solution has been investigated both theoretically and numerically. Some authors used the method of separating variables to obtain analytical solutions of the multi-term time fractional diffusion-wave equations and the multi-term time fractional diffusion equation \cite{Gejji,Luchko,Jiang1}. Numerical solutions for the multi-term time fractional diffusion-wave equation can be found in \cite{Dehghan,Ye15} and for the multi-term time fractional diffusion equation can be found in \cite{Qin17,Jin,Zheng}, respectively. There is also some research on the numerical solution of the multi-term time fractional diffusion equation, of which the indices belong to $(0,2)$ or greater than 2 \cite{Ford,Liu13}.

Different to the general multi-term time fractional diffusion equation, the new model (\ref{e04}) not only has a multi-term time derivative, of which the fractional order indices are from 0 to 2,  but also possesses a special time fractional operator on the spatial derivative, which is challenging to approximate. Although there is some literature \cite{a12,Zheng12,Ming16} involving the exact solution of the generalized Oldroyd-B fluid, the solution is typically given in series form with special functions, such as the Fox $H$-function or the multivariate Mittag-Leffler function, and both of these functions are difficult to express explicitly. Therefore, numerical solution of (\ref{e04}) is a promising tool to provide insight on the behaviour of the model. In \cite{Baz}, Bazhlekova and Bazhlekov presented a finite difference method to solve the viscoelastic flow of a generalized Oldroyd-B fluid (\ref{e01}). They utilised the Gr\"unwald-Letnikov formula to approximate the Riemann-Liouville time fractional derivative, which has first order accuracy, however no theoretical analysis was given. Recently, Feng et al. \cite{Feng17} proposed a finite difference method for the generalized fractional Oldroyd-B fluid (\ref{e01}) between two rigid plates and gave the stability and convergence analysis, which has low order accuracy as well. To the best of the authors' knowledge, there is no literature reported on the numerical solution of Eqs.(\ref{e02}) and (\ref{e03}). Therefore, the numerical solution of Eq.(\ref{e04}) has also not appeared. For the two kinds of time fractional derivatives in the L.H.S. of Eq.(\ref{e04}), the so-called $L1$ or $L2$ scheme can be used for approximation. For the coupled operator (time fractional operator on the spatial derivative) in the R.H.S. of Eq.(\ref{e04}), few techniques can be applied. As Eq.(\ref{e04}) involves these terms simultaneously, the derivation of the numerical solution becomes difficult and it is more challenging to establish the theoretical analysis. The main contributions of this paper are as follows:

\begin{itemize}
\item We propose two new different finite difference schemes to approximate the coupled operator ${D_t^\beta}\frac{\partial^2 u}{\partial {x^2}}$, in which the mixed L scheme is used to discretise the equation at mesh point $(x_i,t_{n-\frac{1}{2}})$ directly. We also establish the $L2$ scheme for the term ${D_t^\gamma}u(x,t)$ with first order accuracy. In addition, we give an important and useful lemma, which can be extended to other multi-term time fractional diffusion problems;
\item We derive two different finite difference schemes for problem (\ref{e04}) with accuracy $O(\tau+h^2)$ and $O(\tau^{\min\{3-\gamma_{s},2-\alpha_{q},2-\beta\}}$ $+h^2)$, respectively and establish the stability and convergence analysis. We prove our method is unconditionally stable and convergent under the discrete $H^1$ norm;
\item Our numerical methods are robust and flexible, which can be used to deal with problem (\ref{e04}) with different initial and boundary conditions,  for which an analytical solution may be not feasible;
\item Our numerical solution is more general and can be used to solve other time fractional diffusion problems, such as the generalized Oldroyd-B fluid model with or without a magnetic field effect, the generalized Maxwell fluid model, the generalized second grade fluid model and the generalized Burgers' fluid model.
\end{itemize}

The outline of the paper is as follows. In Section 2, some preliminary knowledge is given, in which two numerical schemes to discretise the time fractional derivative are proposed. In Section 3, we develop the finite difference schemes for Eq.(\ref{e04}). We proceed with the proof of the stability and convergence of the scheme using the energy method and discuss the solvability of the numerical scheme in Section 4. In section 5, we present two numerical examples to demonstrate the effectiveness of our method and some conclusions are summarised.

\section{Preliminary knowledge}

For convenience, in the subsequent sections, we suppose that $C, C_1, C_2, \ldots$ are positive constants, whose values will be implicitly determined by the surrounding context.

Firstly, in the interval $[0,L]$, we take the mesh points $x_i=ih$,  $i = 0, 1,\cdots,M$, and $t_n = n\tau$, $n = 0, 1,\cdots, N$, where $h =L/M$, $ \tau= T/N$ are the uniform spatial step size and temporal step size, respectively. Denote $\Omega_{\tau}\equiv \{t_n|~0\leq n\leq N\}$ and $\Omega_{h}\equiv \{x_i|~0\leq i\leq M\}$. Define the grid function $u_i^n=u(x_i,t_n)$ and $f_i^n=f(x_i,t_n)$.  We introduce the following notations:
\begin{align*}
\nabla_tu_i^{n}=\frac{u_i^n-u_i^{n-1}}{\tau},\quad u_i^{n-\frac{1}{2}}=\frac{u_i^n+u_i^{n-1}}{2},\quad
\nabla_xu_i^{n}=\frac{u_i^n-u_{i-1}^{n}}{h},\quad\delta_x^2u_i^n=\frac{u_{i-1}^n-2u_i^n+u_{i+1}^n}{h^2}.
\end{align*}
Denote
\begin{align*}
\mathcal{V}_h=\{v~|~v~ \text{is~a~grid~function~on}~\Omega_{h}~\text{and}~v_0=v_M=0\}.
\end{align*}
For any $\chi,v\in \mathcal{V}_h$, we define the following discrete inner products and induced norms:
\begin{align*}
&(\chi,v)=h\sum_{i=1}^{M-1}\chi_iv_i,\quad \langle\nabla_x\chi,\nabla_xv\rangle=h\sum_{i=1}^{M}\nabla_x\chi_i\cdot\nabla_xv_i,\\
&||v||_0=\sqrt{(v,v)},\quad  ||v||_{\infty}=\max\limits_{0\leq i\leq M} |v_i|,\\
&|v|_1=\sqrt{\langle\nabla_xv,\nabla_xv\rangle}, \quad ||v||_1=\sqrt{a_2||v||_0^2+a_3|v|_1^2}.
\end{align*}
It is straightforward to check that
\begin{align}\label{e07}
(\delta_x^2v^k,v^n)&=-\langle\nabla_xv^k,\nabla_xv^n\rangle,\\\label{e08}
(\delta_x^2v^k,\nabla_tv^n)&=-\frac{1}{\tau}\langle\nabla_xv^k,\nabla_xv^n-\nabla_xv^{n-1}\rangle
=-\langle\nabla_xv^k,\nabla_t(\nabla_xv^n)\rangle.
\end{align}
To discretise the time fractional derivative ${D_t^\gamma}u(x,t)$ $(1<\gamma<2)$ at $(x_i,t_n)$, we have
\begin{align*}
{D_t^\gamma} u(x_i,t_n)=&\frac{1}{\Gamma (2-\gamma)}\int_{0}^{t_n}  (t_n-s)^{1-\gamma} \frac{\partial^2
{u(x_i,s)}}{\partial {s^2}}d s\\
=&\frac{1}{\Gamma (2-\gamma)}\sum_{k=1}^{n}\int_{t_{k-1}}^{t_k}  (t_n-s)^{1-\gamma} \frac{\partial^2
{u(x_i,s)}}{\partial {s^2}}d s\\
=&\frac{1}{\Gamma (2-\gamma)}\sum_{k=1}^{n}\int_{t_{k-1}}^{t_k}  (t_n-s)^{1-\gamma}\cdot\frac{u_i^k-2u_i^{k-1}+u_i^{k-2}}{\tau^2} d s+r^n\\
=&\frac{\tau^{1-\gamma}}{\Gamma(3-\gamma)}\Big[a_0^{(\gamma)}\nabla_tu_i^{n}
-\sum_{k=1}^{n-1}(a_{n-k-1}^{(\gamma)}-a_{n-k}^{(\gamma)})\nabla_tu_i^{k}-
a_{n-1}^{(\gamma)}\frac{\partial u(x_i,0)}{\partial t}\Big]+r^n,
\end{align*}
where $u_i^{-1}=u_i^0-\tau\frac{\partial u(x_i,0)}{\partial t}$ and $a_k^{(\gamma)}=(k+1)^{2-\gamma}-k^{2-\gamma}$, $k=0,1,2,\ldots$ For the truncation error $r^n$, we have \cite{Li11}
\begin{align*}
|r^n|\leq \frac{CT^{2-\gamma}}{\Gamma(3-\gamma)}\max\limits_{0\leq t\leq T}\bigg|\frac{\partial^3{u(x_i,t)}}{\partial t^3} \bigg|\cdot \tau+O(\tau^2).
\end{align*}
Then we can obtain the discrete scheme for the time fractional derivative ${D_t^\gamma}u(x,t)$ at mesh points $(x_i,t_n)$
\begin{align}\label{e10}
{D_t^\gamma}u(x_i,t_{n})=\frac{\tau^{1-\gamma}}{\Gamma(3-\gamma)}\Big[a_0^{(\gamma)}\nabla_tu_i^{n}
-\sum_{k=1}^{n-1}(a_{n-k-1}^{(\gamma)}-a_{n-k}^{(\gamma)})\nabla_tu_i^{k}-
a_{n-1}^{(\gamma)}\frac{\partial u(x_i,0)}{\partial t}\Big]+R_1,
\end{align}
where $|R_1|\leq C\tau$.

To discretise the time fractional derivative ${D_t^\gamma}u(x,t)$ $(1<\gamma<2)$  at $(x_i,t_{n-\frac{1}{2}})$, we have the following so-called $L2$ formula \cite{Sun06}
\begin{align}\label{e11}
{D_t^\gamma}u(x_i,t_{n-\frac{1}{2}})=\frac{\tau^{1-\gamma}}{\Gamma(3-\gamma)}
\Big[a_0^{(\gamma)}\nabla_tu_i^{n}
-\sum_{k=1}^{n-1}(a_{n-k-1}^{(\gamma)}-a_{n-k}^{(\gamma)})\nabla_tu_i^{k}-
a_{n-1}^{(\gamma)}\frac{\partial u(x_i,0)}{\partial t}\Big]+R_2,
\end{align}
where $|R_2|\leq C\tau^{3-\gamma}$.
\begin{lemma}\label{lm2}
For $1<\gamma<2$, define $a_k^{(\gamma)}=(k+1)^{2-\gamma}-k^{2-\gamma}$, $k=0,1,2,\ldots,n$ and vector $S=\{S_1,S_2,S_3,\ldots,S_N\}$ and constant $P$, then it holds that
\begin{align*}
&\frac{\tau^{1-\gamma}}{\Gamma(3-\gamma)}\sum_{n=1}^{N}\Big[a_0^{(\gamma)}S_n-
\sum_{k=1}^{n-1}(a_{n-k-1}^{(\gamma)}-a_{n-k}^{(\gamma)})S_k-
a_{n-1}^{(\gamma)}P\Big]S_n\\
\geq &\frac{T^{1-\gamma}}{2\Gamma(2-\gamma)}\sum_{n=1}^{N}S_n^2
-\frac{T^{2-\gamma}}{2\tau\Gamma(3-\gamma)}P^2,~N=1,2,3,\ldots
\end{align*}
\end{lemma}
\begin{proof}
See \cite{Sun06}.
\end{proof}
To discretise the time fractional derivative ${D_t^\beta}u(x,t)$ $(0<\beta<1)$ at $(x_i,t_{n})$, we have the following so-called $L1$ formula \cite{Sun06}
\begin{align}
{D_t^\beta}u(x_i,t_{n})&=\frac{\tau^{-\beta}}{\Gamma(2-\beta)}
\Big[d_0^{(\beta)}u_i^{n}
-\sum_{k=1}^{n-1}(d_{n-k-1}^{(\beta)}-d_{n-k}^{(\beta)}) u_i^{k}-
d_{n-1}^{(\beta)} u_i^{0}\Big]+R_3\nonumber\\\label{e12}
&=\frac{\tau^{1-\beta}}{\Gamma(2-\beta)}
\sum_{k=1}^{n}d_{n-k}^{(\beta)} \nabla_tu_i^{k}+R_3,
\end{align}
where $d_k^{(\beta)}=(k+1)^{1-\beta}-k^{1-\beta}$, $k=0,1,2,\ldots,n$ and $|R_3|\leq C\tau^{2-\beta}$. It is straightforward to derive the following lemma on the properties of $d_k^{(\beta)}$ \cite{Liu07}.
\begin{lemma}\label{lm3}
For $0<\beta<1$, define $d_k^{(\beta)}=(k+1)^{1-\beta}-k^{1-\beta}$, $k=0,1,2,\ldots$ then
\begin{itemize}
\item[1.] $d_k^{(\beta)}>0$, $d_0^{(\beta)}=1$, $d_k^{(\beta)}>d_{k+1}^{(\beta)}$, ${\lim\limits_{k\to \infty}}d_k^{(\beta)}=0$,
\item[2.] $\sum\limits_{k=0}^{n-1}(d_k^{(\beta)}-d_{k+1}^{(\beta)})+d_{n}^{(\beta)}=1$,
\item[3.] $d_{k+1}^{(\beta)}-2d_k^{(\beta)}+d_{k-1}^{(\beta)}\geq0$,~$k\geq 1$.
\end{itemize}
\end{lemma}
Since
\begin{align*}
\frac{\partial^2u(x_i,t_n)}{\partial x^2}=\delta_x^2u_i^n-\frac{h^2}{12}\frac{\partial^4u(\xi_i,t_n)}{\partial x^4},
\end{align*}
then we have
\begin{align}
{D_t^\beta}\frac{\partial^2u(x_i,t_n)}{\partial x^2}
&=\frac{\tau^{-\beta}}{\Gamma(2-\beta)}\Big[d_0^{(\beta)}\delta_x^2u_i^n-
\sum_{k=1}^{n-1}(d_{n-k-1}^{(\beta)}-d_{n-k}^{(\beta)})\delta_x^2u_i^k-
d_{n-1}^{(\beta)}\delta_x^2 u_i^0\Big]+R_4\nonumber\\\label{e13}
&=\frac{\tau^{1-\beta}}{\Gamma(2-\beta)}
\sum_{k=1}^{n}d_{n-k}^{(\beta)} \nabla_t(\delta_x^2u_i^{k})+R_4,
\end{align}
where $|R_4|\leq C(\tau^{2-\beta}+h^2)$.
\begin{lemma}\label{lm4}
For $0<\beta<1$, it holds that
\begin{align*}
\frac{\tau^{1-\beta}}{\Gamma(2-\beta)}
\sum_{k=1}^{n}d_{n-k}^{(\beta)} \Big(\nabla_t(\delta_x^2u^{k}),\nabla_tu^n\Big)=-\frac{\tau^{1-\beta}}{\Gamma(2-\beta)}
\sum_{k=1}^{n}d_{n-k}^{(\beta)} \Big\langle\nabla_t(\nabla_xu^{k}),\nabla_t(\nabla_xu^{n})\Big\rangle.
\end{align*}
\end{lemma}
\begin{proof}
Combining (\ref{e08}) and (\ref{e13}), we obtain
\begin{align*}
&\frac{\tau^{1-\beta}}{\Gamma(2-\beta)}
\sum_{k=1}^{n}d_{n-k}^{(\beta)} \Big(\nabla_t(\delta_x^2u^{k}),\nabla_tu^n\Big)\\
=&\frac{\tau^{-\beta}}{\Gamma(2-\beta)}\Big[d_0^{(\beta)}\Big(\delta_x^2u^n,\nabla_tu^n\Big)-
\sum_{k=1}^{n-1}(d_{n-k-1}^{(\beta)}-d_{n-k}^{(\beta)})\Big(\delta_x^2u^k,\nabla_tu^n\Big)-
d_{n-1}^{(\beta)}\Big(\delta_x^2 u^0,\nabla_tu^n\Big)\Big]\\
=&-\frac{\tau^{-\beta}}{\Gamma(2-\beta)}\Big[d_0^{(\beta)}\Big\langle\nabla_xu^n,\nabla_t(\nabla_xv^n)\Big\rangle-
\sum_{k=1}^{n-1}(d_{n-k-1}^{(\beta)}-d_{n-k}^{(\beta)})\Big\langle\nabla_xu^k,\nabla_t(\nabla_xv^n)\Big\rangle-
d_{n-1}^{(\beta)}\Big\langle\nabla_xu^0,\nabla_t(\nabla_xv^n)\Big\rangle\Big]\\
=&-\frac{\tau^{-\beta}}{\Gamma(2-\beta)}\Big\langle d_0^{(\beta)}\nabla_xu^n-\sum_{k=1}^{n-1}(d_{n-k-1}^{(\beta)}-d_{n-k}^{(\beta)})
\nabla_xu^k-d_{n-1}^{(\beta)}\nabla_xu^0,\nabla_t(\nabla_xv^n)\Big\rangle\\
=&-\frac{\tau^{1-\beta}}{\Gamma(2-\beta)}
\sum_{k=1}^{n}d_{n-k}^{(\beta)} \Big\langle\nabla_t(\nabla_xu^{k}),\nabla_t(\nabla_xu^{n})\Big\rangle.
\end{align*}
\end{proof}

Now we consider the discretization of ${D_t^\beta}u(x,t)$ at grid points $(x_i,t_{n-\frac{1}{2}})$. From (\ref{e12}), we have
\begin{align}\label{e14}
{D_t^\beta}u(x_i,t_{n-\frac{1}{2}})\approx&\frac{1}{2}\Big[{D_t^\beta}u(x_i,t_{n})+{D_t^\beta}u(x_i,t_{n-1})\Big]\nonumber\\
=&\frac{\tau^{1-\beta}}{2\Gamma(2-\beta)}\Big[
\sum_{k=1}^{n}d_{n-k}^{(\beta)} \nabla_tu_i^{k}+\sum_{k=1}^{n-1}d_{n-1-k}^{(\beta)} \nabla_tu_i^{k}\Big]+R_3.
\end{align}
Similarly, we have
\begin{align}\label{e15}
{D_t^\beta}\frac{\partial^2 u(x_i,t_{n-\frac{1}{2}})}{\partial {x^2}}\approx&\frac{1}{2}\Big[{D_t^\beta}\frac{\partial^2u(x_i,t_n)}{\partial x^2}+{D_t^\beta}\frac{\partial^2u(x_i,t_{n-1})}{\partial x^2}\Big]\nonumber\\
=&\frac{\tau^{1-\beta}}{2\Gamma(2-\beta)}\Big[
\sum_{k=1}^{n}d_{n-k}^{(\beta)} \nabla_t(\delta_x^2u_i^{k})+
\sum_{k=1}^{n-1}d_{n-1-k}^{(\beta)} \nabla_t(\delta_x^2u_i^{k})\Big]+R_4.
\end{align}

\begin{lemma}\cite{Marcos}\label{lm7}
Let $\{g_0, g_1,\ldots, g_n,\ldots\}$ be a sequence of real numbers with the
properties
\begin{align*}
g_n\geq 0,\quad g_n-g_{n-1}\leq 0,\quad g_{n+1}-2g_n+g_{n-1}\geq 0.
\end{align*}
Then for any positive integer $M$, and for each vector $[V_1, V_2,\ldots, V_M]$ with $M$ real
entries,
\begin{align*}
\sum_{n=1}^{M}\bigg(\sum_{p=0}^{n-1} g_{p}\,V_{n-p}\bigg)V_n\geq 0.
\end{align*}
\end{lemma}
Now we will prove a very important and useful lemma.
\begin{lemma}\label{lm5}
For $0<\beta<1$, define $d_k^{(\beta)}=(k+1)^{1-\beta}-k^{1-\beta}$, $k=0,1,2,\ldots,n$, then for any positive integer $N$ and vector $\mathbf{Q}=[v^1,v^2,\ldots,v^{N-1},v^{N}]\in R^{N}$, we have
\begin{align}\label{e16}
&\sum_{n=1}^{N}\sum_{k=1}^{n} d_{n-k}^{(\beta)}\,v^k v^n\geq 0,\\\label{e17}
&\sum_{n=1}^{N}\sum_{k=1}^{n} d_{n-k}^{(\beta)}\,v^k v^n+\sum_{n=1}^{N}\sum_{k=1}^{n-1} d_{n-1-k}^{(\beta)}\,v^k v^n\geq 0.
\end{align}
\end{lemma}

\begin{proof}
It is  easy to check that
\begin{align*}
\sum_{n=1}^{N}\sum_{k=1}^{n} d_{n-k}^{(\beta)}\,v^kv^n=\sum_{n=1}^{N}\bigg(\sum_{k=0}^{n-1} d_{k}^{(\beta)}\,v^{n-k}\bigg)v^n.
\end{align*}
Then using Lemmas \ref{lm3} and \ref{lm7}, we have
\begin{align*}
\sum_{n=1}^{N}\bigg(\sum_{k=0}^{n-1} d_{k}^{(\beta)}\,v^{n-k}\bigg)v^n\geq 0,
\end{align*}
i.e.,
\begin{align*}
\sum_{n=1}^{N}\sum_{k=1}^{n} d_{n-k}^{(\beta)}\,v^k v^n\geq 0.
\end{align*}
For the sum in (\ref{e17}), we can rewrite it in the following form
\begin{align*}
\sum_{n=1}^{N}\sum_{k=1}^{n} d_{n-k}^{(\beta)}\,v^k v^n+\sum_{n=1}^{N}\sum_{k=1}^{n-1} d_{n-1-k}^{(\beta)}\,v^k v^n
=\mathbf{Q}A\mathbf{Q}^T,
\end{align*}
where
\begin{align*}
A=\left[\begin{array}{cccccc}
                  d_{0}^{(\beta)}& 0 & 0 &  \cdots & 0 & 0 \\
                  d_{0}^{(\beta)}+d_{1}^{(\beta)} & d_{0}^{(\beta)} & 0 &  \cdots & 0 & 0 \\
                  d_{1}^{(\beta)}+d_{2}^{(\beta)} & d_{0}^{(\beta)}+d_{1}^{(\beta)} & d_{0}^{(\beta)} &  \cdots & 0 & 0 \\
                  d_{2}^{(\beta)}+d_{3}^{(\beta)} & d_{1}^{(\beta)}+d_{2}^{(\beta)} & d_{0}^{(\beta)}+d_{1}^{(\beta)} &  \cdots & 0 & 0 \\
                  \vdots & \vdots& \vdots & \ddots & \vdots & \vdots \\
                  d_{N-3}^{(\beta)}+d_{N-2}^{(\beta)} & d_{N-4}^{(\beta)}+d_{N-3}^{(\beta)} & d_{N-5}^{(\beta)}+d_{N-4}^{(\beta)} & \cdots & d_{0}^{(\beta)} & 0 \\
                  d_{N-2}^{(\beta)}+d_{N-1}^{(\beta)} & d_{N-3}^{(\beta)}+d_{N-2}^{(\beta)}  & d_{N-4}^{(\beta)}+d_{N-3}^{(\beta)}  & \cdots & d_{0}^{(\beta)}+d_{1}^{(\beta)} &d_{0}^{(\beta)}\\
                  \end{array}
                  \right].
\end{align*}
We can notice that to prove (\ref{e17}) is equivalent to proving the matrix $A$ is positive definite. Therefore we only need to prove $H_{N}=\frac{A+A^T}{2}$ is positive definite \cite{Qua}. $H_{N}$ is a real symmetric Toeplitz matrix and has the form
\begin{align*}
H_{N}=\frac{1}{2}\left[\begin{array}{cccccc}
                  2d_{0}^{(\beta)}& d_{0}^{(\beta)}+d_{1}^{(\beta)} & d_{1}^{(\beta)}+d_{2}^{(\beta)} &  \cdots &  d_{N-3}^{(\beta)}+d_{N-2}^{(\beta)} & d_{N-2}^{(\beta)}+d_{N-1}^{(\beta)} \\
                  d_{0}^{(\beta)}+d_{1}^{(\beta)} & 2d_{0}^{(\beta)} & d_{0}^{(\beta)}+d_{1}^{(\beta)} &  \cdots &  d_{N-4}^{(\beta)}+d_{N-3}^{(\beta)} &  d_{N-3}^{(\beta)}+d_{N-2}^{(\beta)} \\
                  d_{1}^{(\beta)}+d_{2}^{(\beta)} & d_{0}^{(\beta)}+d_{1}^{(\beta)} & 2d_{0}^{(\beta)} &  \cdots & d_{N-5}^{(\beta)}+d_{N-4}^{(\beta)} &  d_{N-4}^{(\beta)}+d_{N-3}^{(\beta)} \\
                  d_{2}^{(\beta)}+d_{3}^{(\beta)} & d_{1}^{(\beta)}+d_{2}^{(\beta)} & d_{0}^{(\beta)}+d_{1}^{(\beta)} &  \cdots &d_{N-6}^{(\beta)}+d_{N-5}^{(\beta)}  & d_{N-5}^{(\beta)}+d_{N-4}^{(\beta)} \\
                  \vdots & \vdots& \vdots & \ddots & \vdots & \vdots \\
                  d_{N-3}^{(\beta)}+d_{N-2}^{(\beta)} & d_{N-4}^{(\beta)}+d_{N-3}^{(\beta)} & d_{N-5}^{(\beta)}+d_{N-4}^{(\beta)} & \cdots & 2d_{0}^{(\beta)} & d_{0}^{(\beta)}+d_{1}^{(\beta)} \\
                  d_{N-2}^{(\beta)}+d_{N-1}^{(\beta)} & d_{N-3}^{(\beta)}+d_{N-2}^{(\beta)}  & d_{N-4}^{(\beta)}+d_{N-3}^{(\beta)}  & \cdots & d_{0}^{(\beta)}+d_{1}^{(\beta)} &2d_{0}^{(\beta)}\\
                  \end{array}
                  \right].
\end{align*}
In the following, we will prove $\det(H_N)>0$. It is straightforward to verify that $\det(H_1)=d_{0}^{(\beta)}>0$, $\det(H_2)=(d_{0}^{(\beta)})^2-\frac{(d_{0}^{(\beta)}+d_{1}^{(\beta)})^2}{4}>0$. For a finite integer $N$, we can explicitly calculate the value of $\det(H_{N})>0$. When $N$ is a sufficiently large, according to \cite{Bo} (Propositions 10.2 and 10.4), we have
\begin{align*}
\frac{\det(H_N)}{\det(H_{N+1})}>0.
\end{align*}
Then we can conclude that $\det(H_{N+1})>0$. To illustrate this, we give a figure plot of $\frac{\det(H_N)}{\det(H_{N+1})}$  with different $\beta$ and $N$ (see Fig. 1). We can see that $\frac{\det(H_N)}{\det(H_{N+1})}>0$, particularly, when almost $N>250$, $\frac{\det(H_N)}{\det(H_{N+1})}\approx C_{\beta}>0$.
\begin{figure}[htbp]
\centering
\scalebox{0.6}[0.6]{\includegraphics{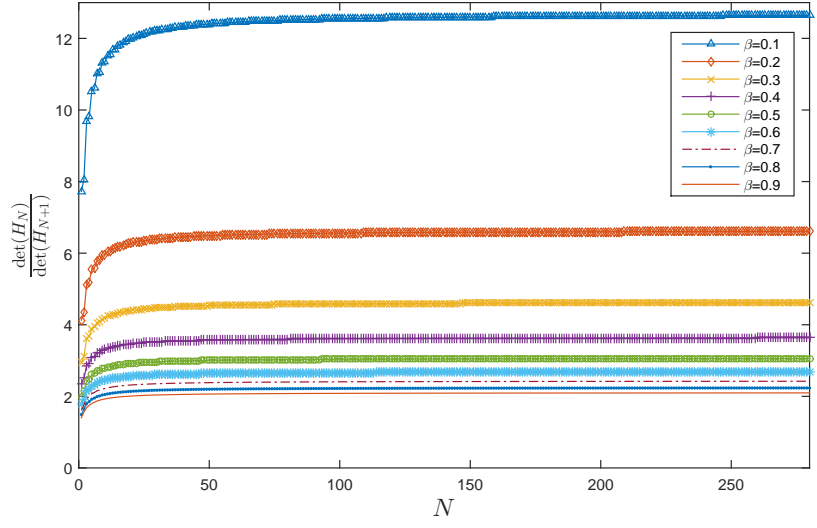}}
\caption{Figure plot of $\frac{\det(H_N)}{\det(H_{N+1})}$  with different $\beta$ and $N$.}
\end{figure}

As matrix $H_i,~i=1,2,\ldots,N$ are the principal minors of matrix $H_{N+1}$ and $\det(H_k)>0,~k=1,2,\ldots,N+1$, then the
real symmetric Toeplitz matrix $H_{N+1}$ is positive definite. The proof is completed.
\end{proof}

To derive the finite difference scheme we also need the following lemma.
\begin{lemma}\cite{Chen12}\label{lm6}
If $u(x,t)\in C_{x,t}^{0,3}(\Omega)$, then we have
\begin{align}\label{e18}
u(x_i,t_{n-\frac{1}{2}})&=\frac{u(x_i,t_n)+u(x_i,t_{n-1})}{2}+O(\tau^2),\\\label{e19}
\frac{\partial }{\partial t}u(x_i,t_{n-\frac{1}{2}})&=\frac{u(x_i,t_n)-u(x_i,t_{n-1})}{\tau}+O(\tau^2).
\end{align}
\end{lemma}
For the discretization of the time fractional derivative ${D_t^\alpha}u(x,t)$ $(0<\alpha<1)$, it is the same with ${D_t^\beta}u(x,t)$.
\section{Derivation  of the numerical schemes}

In this section, we will give two different finite difference schemes of Eq.(\ref{e04}).
\subsection{Scheme I: first order implicit scheme}
Assume that $u(x,t)\in C_{x,t}^{4,3}(\Omega)$, from Eq.(\ref{e04}), we have
\begin{align}
&\sum_{j=1}^{s}b_j\,{D}_t^{\gamma_j}u(x_i,t_n)+{a_1}\frac{\partial u(x_i,t_n)}{\partial t}+\sum_{l=1}^{q}c_l\,{D}_t^{\alpha_l}u(x_i,t_n)+a_2u(x_i,t_n)\nonumber\\\label{e20}
=&{a_3}\frac{\partial^2 u(x_i,t_n)}{\partial {x^2}}
+{a_4}{D_t^\beta}\frac{\partial^2 u(x_i,t_n)}{\partial {x^2}}+f(x_i,t_n).
\end{align}
Using Eqs.(\ref{e10}), (\ref{e12}) and (\ref{e13}), we obtain
\begin{align}
&\sum_{j=1}^{s}b_j\mu_{1,j}\Big[a_0^{(\gamma_j)}\nabla_tu_{i}^{n}-
\sum_{k=1}^{n-1}(a_{n-k-1}^{(\gamma_j)}-a_{n-k}^{(\gamma_j)})\nabla_tu_{i}^{k}-
a_{n-1}^{(\gamma_j)}\phi_2(x_i)\Big]\nonumber\\
+&a_1\nabla_tu_{i}^{n}+\sum_{l=1}^{q}c_l\mu_{2,l}
\sum_{k=1}^{n}d_{n-k}^{(\alpha_l)} \nabla_tu_i^{k}
+a_2u_{i}^{n}\nonumber\\\label{e21}
=&a_3\delta_x^2u_{i}^{n}+a_4\mu_3
\sum_{k=1}^{n}d_{n-k}^{(\beta)} \nabla_t(\delta_x^2u_i^{k})+f_i^n+R_{1,i}^n,
\end{align}
where $\mu_{1,j}=\frac{\tau^{1-\gamma_j}}{\Gamma(3-\gamma_j)}$,
 $\mu_{2,l}=\frac{\tau^{1-\alpha_l}}{\Gamma(2-\alpha_l)}$, $\mu_3=\frac{\tau^{1-\beta}}{\Gamma(2-\beta)}$ and $|R_{1,i}^n|\leq C(\tau+h^2)$, in which $C$ is independent of $\tau$ and $h$. Then, omitting the error term and denoting $U_{i}^{n}$ as the numerical approximation to $u_{i}^{n}$, the implicit finite difference scheme for Eq.(\ref{e04}) at point $(x_i,t_n)$ is given by
\begin{align}
&\sum_{j=1}^{s}b_j\mu_{1,j}\Big[a_0^{(\gamma_j)}\nabla_tU_{i}^{n}-
\sum_{k=1}^{n-1}(a_{n-k-1}^{(\gamma_j)}-a_{n-k}^{(\gamma_j)})\nabla_tU_{i}^{k}-
a_{n-1}^{(\gamma_j)}\phi_2(x_i)\Big]\nonumber\\
+&a_1\nabla_tU_{i}^{n}+\sum_{l=1}^{q}c_l\mu_{2,l}
\sum_{k=1}^{n}d_{n-k}^{(\alpha_l)} \nabla_tU_i^{k}
+a_2U_{i}^{n}\nonumber\\\label{e22}
=&a_3\delta_x^2U_{i}^{n}+a_4\mu_3
\sum_{k=1}^{n}d_{n-k}^{(\beta)} \nabla_t(\delta_x^2U_i^{k})+f_i^n,
\end{align}
with initial and boundary conditions
\begin{align*}
U^0_i=\phi_1(x_i),\quad 0\leq i\leq M, \quad U_0^n=U_M^n=0,\quad 1\leq n\leq N.
\end{align*}

\subsection{Scheme II: mixed L scheme}

Assume that $u(x,t)\in C_{x,t}^{4,3}(\Omega)$, from Eq.(\ref{e04}), we have
\begin{align}
&\sum_{j=1}^{s}b_j\,{D}_t^{\gamma_j}u(x_i,t_{n-\frac{1}{2}})+{a_1}\frac{\partial u(x_i,t_{n-\frac{1}{2}})}{\partial t}+\sum_{l=1}^{q}c_l\,{D}_t^{\alpha_l}u(x_i,t_{n-\frac{1}{2}})+a_2u(x_i,t_{n-\frac{1}{2}})\nonumber\\\label{e23}
=&{a_3}\frac{\partial^2 u(x_i,t_{n-\frac{1}{2}})}{\partial {x^2}}
+{a_4}{D_t^\beta}\frac{\partial^2 u(x_i,t_{n-\frac{1}{2}})}{\partial {x^2}}+f(x_i,t_{n-\frac{1}{2}}).
\end{align}
Applying  Eqs.(\ref{e11}), (\ref{e14}) and (\ref{e15}) and Lemma \ref{lm6}, we have
\begin{align}
&\sum_{j=1}^{s}b_j\mu_{1,j}\Big[a_0^{(\gamma_j)}\nabla_tu_{i}^{n}-
\sum_{k=1}^{n-1}(a_{n-k-1}^{(\gamma_j)}-a_{n-k}^{(\gamma_j)})\nabla_tu_{i}^{k}-
a_{n-1}^{(\gamma_j)}\phi_2(x_i)\Big]+a_1\nabla_tu_{i}^{n}\nonumber\\
+&\sum_{l=1}^{q}\frac{c_l\mu_{2,l}}{2}\Big[
\sum_{k=1}^{n}d_{n-k}^{(\alpha_l)} \nabla_tu_i^{k}+\sum_{k=1}^{n-1}d_{n-1-k}^{(\alpha_l)} \nabla_tu_i^{k}\Big]
+a_2u_{i}^{n-\frac{1}{2}}\label{e24}\\\nonumber
=&a_3\delta_x^2u_{i}^{n-\frac{1}{2}}
+\frac{a_4\mu_3}{2}\Big[
\sum_{k=1}^{n}d_{n-k}^{(\beta)} \nabla_t(\delta_x^2u_i^{k})+
\sum_{k=1}^{n-1}d_{n-1-k}^{(\beta)} \nabla_t(\delta_x^2u_i^{k})\Big]+f_i^{n-\frac{1}{2}}+R_{2,i}^n,
\end{align}
where $|R_{2,i}^n|\leq C(\tau^{\min\{3-\gamma_s,2-\alpha_q,2-\beta\}}+h^2)$. Then, omitting the error term, we obtain the mixed L finite difference scheme for Eq.(\ref{e04}) at point $(x_i,t_{n-\frac{1}{2}})$
\begin{align}
&\sum_{j=1}^{s}b_j\mu_{1,j}\Big[a_0^{(\gamma_j)}\nabla_tU_{i}^{n}-
\sum_{k=1}^{n-1}(a_{n-k-1}^{(\gamma_j)}-a_{n-k}^{(\gamma_j)})\nabla_tU_{i}^{k}-
a_{n-1}^{(\gamma_j)}\phi_2(x_i)\Big]+a_1\nabla_tU_{i}^{n}\nonumber\\
+&\sum_{l=1}^{q}\frac{c_l\mu_{2,l}}{2}\Big[
\sum_{k=1}^{n}d_{n-k}^{(\alpha_l)} \nabla_tU_i^{k}+\sum_{k=1}^{n-1}d_{n-1-k}^{(\alpha_l)} \nabla_tU_i^{k}\Big]
+a_2U_{i}^{n-\frac{1}{2}}\label{e25}\\\nonumber
=&a_3\delta_x^2U_{i}^{n-\frac{1}{2}}
+\frac{a_4\mu_3}{2}\Big[
\sum_{k=1}^{n}d_{n-k}^{(\beta)} \nabla_t(\delta_x^2U_i^{k})+
\sum_{k=1}^{n-1}d_{n-1-k}^{(\beta)} \nabla_t(\delta_x^2U_i^{k})\Big]+f_i^{n-\frac{1}{2}}.
\end{align}

\begin{Rem}
Compared to scheme I, scheme II has high order accuracy. However, more terms are added in the scheme II.
\end{Rem}
\section{Theoretical analysis }%
\subsection{Solvability}
Firstly, we discuss the solvability of the finite difference scheme (\ref{e22}).
\begin{thm}\label{thm1}
The finite difference scheme (\ref{e22}) is uniquely solvable.
\end{thm}
\begin{proof}
At each time level, the coefficient matrix $B$ is linear tridiagonal
\begin{equation*}
B=\left[\begin{array}{cccccc}
d_1+2{d_2} & -{d_2} & 0& \cdots & 0&0 \\
-{d_2} & d_1+2{d_2} &-{d_2}&  \cdots & 0 &0\\
0 &  -{d_2}&d_1+2{d_2}&  \cdots & 0 &0 \\
\vdots & \vdots & \vdots & \ddots &\vdots &\vdots\\
0 & 0 & 0 & \cdots &d_1+2{d_2} &-{d_2}\\
0 & 0& 0& \cdots &-{d_2}& d_1+2{d_2} \\
\end{array}\right],
\end{equation*}
where $d_1=\sum\limits_{j=1}^{s}\frac{b_j\mu_{1,j}}{\tau}+\frac{a_1}{\tau}+\sum\limits_{l=1}^{q}\frac{c_l\mu_{2,l}}{\tau}+a_2>0$ and $d_2=\frac{a_3}{h^2}+\frac{a_4\mu_3}{\tau h^2}>0$.
Then $B$ is a strictly diagonally dominant matrix. Therefore $B$ is nonsingular, which means that the numerical scheme (\ref{e22}) is uniquely solvable.
\end{proof}

The solvability of the finite difference scheme (\ref{e25}) is similar.
\subsection{Stability}
Here, we will analyze the stability of the schemes (\ref{e22}) and (\ref{e25}) using the energy method.
\begin{thm}\label{thm2}
The implicit finite difference scheme (\ref{e22}) is unconditionally stable and it holds that
\begin{align*}
||U^{N}||_1^2\leq ||U^{0}||_1^2+\sum_{j=1}^{s}\frac{b_jT^{2-\gamma_j}}{\Gamma(3-\gamma_j)}||\phi_2||_0^2+\frac{T}{2\varepsilon_0}
\max\limits_{1\leq n\leq N}||f^n||_0^2,
\end{align*}
where $\varepsilon_0=\sum\limits_{j=1}^{s}\frac{b_j T^{1-\gamma_j}}{2\Gamma(2-\gamma_j)}+a_1$ and $U^N=[U_{1}^{N},U_{2}^{N},\dots,U_{M-1}^{N}]^T$ is the solution vector of (\ref{e22}).
\end{thm}
\begin{proof}
Multiplying Eq.(\ref{e22}) by $h\tau \nabla_tU_i^{n}$ and summing $i$ from 1 to $M-1$ and $n$ from 1 to $N$, we obtain
\begin{align}
&\tau\sum_{j=1}^{s}b_j\mu_{1,j}\sum_{n=1}^{N}\sum_{i=1}^{M-1}h\Big[a_0^{(\gamma_j)}\nabla_tU_i^{n}-
\sum_{k=1}^{n-1}(a_{n-k-1}^{(\gamma_j)}-a_{n-k}^{(\gamma_j)})\nabla_tU_i^{k}-
a_{n-1}^{(\gamma_j)}\phi_2(x_i)\Big]\nabla_tU_i^{n}
\nonumber\\+&a_1\tau\sum_{n=1}^{N}\sum_{i=1}^{M-1}h(\nabla_tU_i^{n})^2
+\tau\sum_{l=1}^{q}c_l\mu_{2,l}\sum_{n=1}^{N}\sum_{i=1}^{M-1}h\sum_{k=1}^{n}d_{n-k}^{(\alpha_l)} \nabla_tU_i^{k}\nabla_tU_i^{n}
+a_2\tau\sum_{n=1}^{N}\sum_{i=1}^{M-1}hU_i^n\nabla_tU_i^{n}\nonumber\\\label{e26}
=&a_3\tau\sum_{n=1}^{N}\sum_{i=1}^{M-1}h\delta_x^2U_i^n\nabla_tU_i^{n}
+a_4\mu_3\tau\sum_{n=1}^{N}\sum_{i=1}^{M-1}h\sum_{k=1}^{n}d_{n-k}^{(\beta)} \nabla_t(\delta_x^2U_i^{k})\nabla_tU_i^{n}
+\tau\sum_{n=1}^{N}\sum_{i=1}^{M-1}hf_i^n\nabla_tU_i^{n}.
\end{align}
Using Lemma \ref{lm2}, we have
\begin{align}
&\tau\sum_{j=1}^{s}b_j\mu_{1,j}\sum_{n=1}^{N}\sum_{i=1}^{M-1}h\Big[a_0^{(\gamma_j)}\nabla_tU_i^{n}-
\sum_{k=1}^{n-1}(a_{n-k-1}^{(\gamma_j)}-a_{n-k}^{(\gamma_j)})\nabla_tU_i^{k}-
a_{n-1}^{(\gamma_j)}\phi_2(x_i)\Big]\nabla_tU_i^{n}\nonumber\\
\geq& \sum_{j=1}^{s}b_j\bigg(\frac{\tau T^{1-\gamma_j}}{2\Gamma(2-\gamma_j)}\sum_{n=1}^{N}\sum_{i=1}^{M-1}h (\nabla_tU_i^{n})^2
-\frac{ T^{2-\gamma_j}}{2\Gamma(3-\gamma_j)}\sum_{i=1}^{M-1}h\phi_2(x_i)^2\bigg)\nonumber\\\label{e27}
=&\sum_{j=1}^{s}b_j\frac{\tau T^{1-\gamma_j}}{2\Gamma(2-\gamma_j)}\sum_{n=1}^{N}||\nabla_tU^{n}||_0^2
-\sum_{j=1}^{s}b_j\frac{T^{2-\gamma_j}}{2\Gamma(3-\gamma_j)}||\phi_2||_0^2.
\end{align}
For the second term, we have
\begin{align}\label{e28}
a_1\tau\sum_{n=1}^{N}\sum_{i=1}^{M-1}h(\nabla_tU_i^{n})^2=a_1\tau\sum_{n=1}^{N}||\nabla_tU^{n}||_0^2.
\end{align}
Using (\ref{e16}), we obtain
\begin{align}\label{e29}
&\tau\sum_{l=1}^{q}c_l\mu_{2,l}\sum_{n=1}^{N}\sum_{i=1}^{M-1}h\sum_{k=1}^{n}d_{n-k}^{(\alpha_l)} \nabla_tU_i^{k}\nabla_tU_i^{n}
=\tau\sum_{l=1}^{q}c_l\mu_{2,l}\sum_{n=1}^{N}\sum_{k=1}^{n}d_{n-k}^{(\alpha_l)}(\nabla_tU^{k},\nabla_tU^{n})\geq0.
\end{align}
Utilising the inequality $a(a-b)\geq \frac{1}{2}(a^2-b^2)$, we have
\begin{align}\label{e30}
&a_2\tau\sum_{n=1}^{N}\sum_{i=1}^{M-1}hU_i^n\nabla_tU_i^{n}=a_2\sum_{n=1}^{N}(U^n,U^n-U^{n-1})\nonumber\\
\geq&\frac{a_2}{2}\sum_{n=1}^{N}(||U^{n}||_0^2-||U^{n-1}||_0^2)=\frac{a_2}{2}(||U^{N}||_0^2-||U^{0}||_0^2).
\end{align}
Applying (\ref{e08}) and the inequality $a(a-b)\geq \frac{1}{2}(a^2-b^2)$ again, we obtain
\begin{align}
&a_3\tau\sum_{n=1}^{N}\sum_{i=1}^{M-1}h\delta_x^2U_i^n\nabla_tU_i^{n}=a_3\tau\sum_{n=1}^{N}(\delta_x^2U^n,\nabla_tU^{n})
=-a_3\sum_{n=1}^{N}
\langle\nabla_xU^n,\nabla_xU^n-\nabla_xU^{n-1}\rangle\nonumber\\\label{e31}
\leq&-\frac{a_3}{2}\sum_{n=1}^{N}(|U^{n}|_1^2-|U^{n-1}|_1^2)
=\frac{a_3}{2}(|U^{0}|_1^2-|U^{N}|_1^2).
\end{align}
Combining (\ref{e08}), (\ref{e17}) and Lemma \ref{lm4}, we have
\begin{align}
&a_4\mu_3\tau\sum_{n=1}^{N}\sum_{i=1}^{M-1}h\sum_{k=1}^{n}d_{n-k}^{(\beta)} \nabla_t(\delta_x^2U_i^{k})\nabla_tU_i^{n}
=a_4\mu_3\tau\sum_{n=1}^{N}\sum_{k=1}^{n}d_{n-k}^{(\beta)}\Big(\nabla_t(\delta_x^2U^{k}),\nabla_tU^{n}\Big)\nonumber\\\label{e32}
=&-a_4\mu_3\tau\sum_{n=1}^{N}\sum_{k=1}^{n}d_{n-k}^{(\beta)}
\Big\langle\nabla_t(\nabla_xU^{k}),\nabla_t(\nabla_xU^{n})\Big\rangle\leq0.
\end{align}
Using the important inequality $ab\leq \varepsilon a^2+\frac{b^2}{4\varepsilon} (\varepsilon>0)$, we have
\begin{align}
&\tau\sum_{n=1}^{N}\sum_{i=1}^{M-1}hf_i^n\nabla_tU_i^{n}
\leq \tau\varepsilon_0\sum_{n=1}^{N}\sum_{i=1}^{M-1}h (\nabla_tU_i^{n})^2+\frac{\tau}{4\varepsilon_0}
\sum_{n=1}^{N}\sum_{i=1}^{M-1}h(f_i^n)^2\nonumber\\\label{e33}
=&\tau\varepsilon_0\sum_{n=1}^{N}||\nabla_tU^{n}||_0^2+
\frac{\tau}{4\varepsilon_0}
\sum_{n=1}^{N}||f^n||_0^2\leq \tau\varepsilon_0\sum_{n=1}^{N}||\nabla_tU^{n}||_0^2+
\frac{T}{4\varepsilon_0}
\max\limits_{1\leq n\leq N}||f^n||_0^2,
\end{align}
where $\varepsilon_0=\sum\limits_{j=1}^{s}\frac{b_j T^{1-\gamma_j}}{2\Gamma(2-\gamma_j)}+a_1$.
Substituting  (\ref{e27})-(\ref{e33}) into (\ref{e26}), we have
\begin{align*}
&\tau\varepsilon_0\sum_{n=1}^{N}||\nabla_tU^{n}||_0^2
-\sum_{j=1}^{s}\frac{b_jT^{2-\gamma_j}}{2\Gamma(3-\gamma_j)}||\phi_2||_0^2
+\frac{a_2}{2}(||U^{N}||_0^2-||U^{0}||_0^2)\\
\leq& \frac{a_3}{2}(|U^{0}|_1^2-|U^{N}|_1^2)
+\tau\varepsilon_0\sum_{n=1}^{N}||\nabla_tU^{n}||_0^2+\frac{T}{4\varepsilon_0}
\max\limits_{1\leq n\leq N}||f^n||_0^2,
\end{align*}
then rearranging gives
\begin{align}
{a_2}||U^{N}||_0^2+{a_3}|U^{N}|_1^2
\leq {a_2}||U^{0}||_0^2+{a_3}|U^{0}|_1^2\label{e34}
+\sum_{j=1}^{s}\frac{b_jT^{2-\gamma_j}}{\Gamma(3-\gamma_j)}||\phi_2||_0^2+\frac{T}{2\varepsilon_0}
\max\limits_{1\leq n\leq N}||f^n||_0^2,
\end{align}
namely,
\begin{align*}
||U^{N}||_1^2
\leq ||U^{0}||_1^2+\sum_{j=1}^{s}\frac{b_jT^{2-\gamma_j}}{\Gamma(3-\gamma_j)}||\phi_2||_0^2+\frac{T}{2\varepsilon_0}
\max\limits_{1\leq n\leq N}||f^n||_0^2,
\end{align*}
which proves that the scheme (\ref{e22}) is unconditionally stable.
\end{proof}

\begin{thm}\label{thm3}
The implicit finite difference scheme (\ref{e25}) is unconditionally stable and it holds that
\begin{align*}
||\widetilde{U}^{N}||_1^2\leq ||\widetilde{U}^{0}||_1^2+\sum_{j=1}^{s}\frac{b_jT^{2-\gamma_j}}{\Gamma(3-\gamma_j)}||\phi_2||_0^2+\frac{T}{2\varepsilon_0}
\max\limits_{1\leq n\leq N}||f^{n-\frac{1}{2}}||_0^2,
\end{align*}
where $\varepsilon_0=\sum\limits_{j=1}^{s}\frac{b_j T^{1-\gamma_j}}{2\Gamma(2-\gamma_j)}+a_1$ and $\widetilde{U}^N=[U_{1}^{N},U_{2}^{N},\dots,U_{M-1}^{N}]^T$ is the solution vector of (\ref{e25}).
\end{thm}
\begin{proof}

Multiplying Eq.(\ref{e25}) by $h\tau \nabla_tU_i^{n}$ and summing $i$ from 1 to $M-1$ and $n$ from 1 to $N$, we obtain
\begin{align}
&\tau\sum_{j=1}^{s}b_j\mu_{1,j}\sum_{n=1}^{N}\sum_{i=1}^{M-1}h\Big[a_0^{(\gamma_j)}\nabla_tU_i^{n}-
\sum_{k=1}^{n-1}(a_{n-k-1}^{(\gamma_j)}-a_{n-k}^{(\gamma_j)})\nabla_tU_i^{k}-
a_{n-1}^{(\gamma_j)}\phi_2(x_i)\Big]\nabla_tU_i^{n}+a_1\tau\sum_{n=1}^{N}\sum_{i=1}^{M-1}h(\nabla_tU_i^{n})^2
\nonumber\\
+&\tau\sum_{l=1}^{q}\frac{c_l\mu_{2,l}}{2}\sum_{n=1}^{N}\sum_{i=1}^{M-1}h\Big[
\sum_{k=1}^{n}d_{n-k}^{(\alpha_l)} \nabla_tU_i^{k}+\sum_{k=1}^{n-1}d_{n-1-k}^{(\alpha_l)} \nabla_tU_i^{k}\Big]\nabla_tU_i^{n}
+a_2\tau\sum_{n=1}^{N}\sum_{i=1}^{M-1}hU_i^{n-\frac{1}{2}}\nabla_tU_i^{n}\nonumber\\
=&a_3\tau\sum_{n=1}^{N}\sum_{i=1}^{M-1}h\delta_x^2U_i^{n-\frac{1}{2}}\nabla_tU_i^{n}
+\frac{a_4\mu_3\tau}{2}\sum_{n=1}^{N}\sum_{i=1}^{M-1}h\Big[
\sum_{k=1}^{n}d_{n-k}^{(\beta)} \nabla_t(\delta_x^2U_i^{k})+
\sum_{k=1}^{n-1}d_{n-1-k}^{(\beta)} \nabla_t(\delta_x^2U_i^{k})\Big]\nabla_tU_i^{n}\nonumber\\\label{e35}
+&\tau\sum_{n=1}^{N}\sum_{i=1}^{M-1}hf_i^{n-\frac{1}{2}}\nabla_tU_i^{n}.
\end{align}
Using (\ref{e17}), we obtain
\begin{align}
&\tau\sum_{l=1}^{q}\frac{c_l\mu_{2,l}}{2}\sum_{n=1}^{N}\sum_{i=1}^{M-1}h\Big[
\sum_{k=1}^{n}d_{n-k}^{(\alpha_l)} \nabla_tU_i^{k}+\sum_{k=1}^{n-1}d_{n-1-k}^{(\alpha_l)} \nabla_tU_i^{k}\Big]\nabla_tU_i^{n}\nonumber\\\label{e36}
=&\tau\sum_{l=1}^{q}\frac{c_l\mu_{2,l}}{2}\Big[\sum_{n=1}^{N}
\sum_{k=1}^{n}d_{n-k}^{(\alpha_l)} \Big(\nabla_t\widetilde{U}^{k},\nabla_t\widetilde{U}^{n}\Big)+\sum_{n=1}^{N}\sum_{k=1}^{n-1}d_{n-1-k}^{(\alpha_l)} \Big(\nabla_t\widetilde{U}^{k},\nabla_t\widetilde{U}^{n}\Big)\Big]\geq0.
\end{align}
For the fourth term, we have
\begin{align}
&a_2\tau\sum_{n=1}^{N}\sum_{i=1}^{M-1}hU_i^{n-\frac{1}{2}}\nabla_tU_i^{n}
=\frac{a_2}{2}\sum_{n=1}^{N}(\widetilde{U}^n+\widetilde{U}^{n-1},\widetilde{U}^n-\widetilde{U}^{n-1})\nonumber\\\label{e37}
=&\frac{a_2}{2}\sum_{n=1}^{N}(||\widetilde{U}^{n}||_0^2-||\widetilde{U}^{n-1}||_0^2)=\frac{a_2}{2}(||\widetilde{U}^{N}||_0^2-||\widetilde{U}^{0}||_0^2).
\end{align}
For the fifth term, we obtain
\begin{align}\label{e38}
&a_3\tau\sum_{n=1}^{N}\sum_{i=1}^{M-1}h\delta_x^2U_i^{n-\frac{1}{2}}\nabla_tU_i^{n}
=a_3\tau\sum_{n=1}^{N}(\delta_x^2\widetilde{U}^{n-\frac{1}{2}},\nabla_t\widetilde{U}^{n})\nonumber\\
=&-\frac{a_3}{2}\sum_{n=1}^{N}(|\widetilde{U}^{n}|_1^2-|\widetilde{U}^{n-1}|_1^2)
=\frac{a_3}{2}(|\widetilde{U}^{0}|_1^2-|\widetilde{U}^{N}|_1^2).
\end{align}
Combining (\ref{e08}), (\ref{e17}) and Lemma \ref{lm4}, we have
\begin{align}
&\frac{a_4\mu_3\tau}{2}\sum_{n=1}^{N}\sum_{i=1}^{M-1}h\Big[
\sum_{k=1}^{n}d_{n-k}^{(\beta)} \nabla_t(\delta_x^2U_i^{k})+
\sum_{k=1}^{n-1}d_{n-1-k}^{(\beta)} \nabla_t(\delta_x^2U_i^{k})\Big]\nabla_tU_i^{n}\nonumber\\
=&\frac{a_4\mu_3\tau}{2}\sum_{n=1}^{N}\Big[
\sum_{k=1}^{n}d_{n-k}^{(\beta)} \Big(\nabla_t(\delta_x^2\widetilde{U}^{k}),\nabla_t(\widetilde{U}^{n})\Big)+
\sum_{k=1}^{n-1}d_{n-1-k}^{(\beta)}  \Big(\nabla_t(\delta_x^2\widetilde{U}^{k}),\nabla_t(\widetilde{U}^{n})\Big)\Big]\nonumber\\\label{e39}
=& -\frac{a_4\mu_3\tau}{2}\Big[\sum_{n=1}^{N}\sum_{k=1}^{n}d_{n-k}^{(\beta)}
\Big\langle\nabla_t(\nabla_x\widetilde{U}^{k}),\nabla_t(\nabla_x\widetilde{U}^{n})\Big\rangle
+\sum_{n=1}^{N}\sum_{k=1}^{n-1}d_{n-1-k}^{(\beta)}
\Big\langle\nabla_t(\nabla_x\widetilde{U}^{k}),\nabla_t(\nabla_x\widetilde{U}^{n})\Big\rangle\Big]\leq0.
\end{align}
Using the important inequality $ab\leq \varepsilon a^2+\frac{b^2}{4\varepsilon} (\varepsilon>0)$, we have
\begin{align}
\tau\sum_{n=1}^{N}\sum_{i=1}^{M-1}hf_i^{n-\frac{1}{2}}\nabla_tU_i^{n}
&\leq \tau\varepsilon_0\sum_{n=1}^{N}||\nabla_t\widetilde{U}^{n}||_0^2+
\frac{\tau}{4\varepsilon_0}
\sum_{n=1}^{N}||f^{n-\frac{1}{2}}||_0^2\nonumber\\\label{e40}
&\leq \tau\varepsilon_0\sum_{n=1}^{N}||\nabla_t\widetilde{U}^{n}||_0^2+
\frac{T}{4\varepsilon_0}\max\limits_{1\leq n\leq N}||f^{n-\frac{1}{2}}||_0^2,
\end{align}
where $\varepsilon_0=\sum\limits_{j=1}^{s}\frac{b_j T^{1-\gamma_j}}{2\Gamma(2-\gamma_j)}+a_1$.
Substituting  (\ref{e27}), (\ref{e28}) and (\ref{e36})-(\ref{e40}) into (\ref{e35}), we have
\begin{align*}
&\tau\varepsilon_0\sum_{n=1}^{N}||\nabla_t\widetilde{U}^{n}||_0^2
-\sum_{j=1}^{s}\frac{b_jT^{2-\gamma_j}}{2\Gamma(3-\gamma_j)}||\phi_2||_0^2
+\frac{a_2}{2}(||\widetilde{U}^{N}||_0^2-||\widetilde{U}^{0}||_0^2)\\
\leq& \frac{a_3}{2}(|\widetilde{U}^{0}|_1^2-|\widetilde{U}^{N}|_1^2)+\tau\varepsilon_0\sum_{n=1}^{N}||\nabla_t\widetilde{U}^{n}||_0^2
+\frac{T}{4\varepsilon_0}\max\limits_{1\leq n\leq N}||f^{n-\frac{1}{2}}||_0^2,
\end{align*}
then we have
\begin{align*}
{a_2}||\widetilde{U}^{N}||_0^2+{a_3}|\widetilde{U}^{N}|_1^2
\leq {a_2}||\widetilde{U}^{0}||_0^2+{a_3}|\widetilde{U}^{0}|_1^2\label{e39}
+\sum_{j=1}^{s}\frac{b_jT^{2-\gamma_j}}{\Gamma(3-\gamma_j)}||\phi_2||_0^2+\frac{T}{2\varepsilon_0}\max\limits_{1\leq n\leq N}||f^{n-\frac{1}{2}}||_0^2,
\end{align*}
namely,
\begin{align*}
||\widetilde{U}^{N}||_1^2\leq ||\widetilde{U}^{0}||_1^2+\sum_{j=1}^{s}\frac{b_jT^{2-\gamma_j}}{\Gamma(3-\gamma_j)}||\phi_2||_0^2+\frac{T}{2\varepsilon_0}
\max\limits_{1\leq n\leq N}||f^{n-\frac{1}{2}}||_0^2,
\end{align*}
which proves that the scheme (\ref{e25}) is unconditionally stable.
\end{proof}

\subsection{Convergence}

Now we discuss the convergence of the schemes (\ref{e22}) and (\ref{e25}).
\begin{thm}\label{thm4}
Suppose that the solution of problem (\ref{e04})-(\ref{e06}) satisfies $u(x,t)\in C_{x,t}^{4,3}(\Omega)$. Define ${{u}}^n=[u_{1}^{n},u_{2}^{n},$ $\ldots,u_{M-1}^{n}]^T$ as the exact solution vector, $U^n=[U_1^n,U_2^n,\ldots,U_{M-1}^n]^T$ as the numerical solution vector of (\ref{e22}), and $\widetilde{U}^n=[U_1^n,U_2^n,\ldots,U_{M-1}^n]^T$ as the numerical solution vector of (\ref{e25}), respectively.  Then there exists two positive constants $C_1$ and $C_2$ independent of $h$ and $\tau$ such that
\begin{align*}
||{ {u}}^n-U^n||_1&\leq C_1 \sqrt{\frac{TL }{2\varepsilon_0}}(\tau+h^2),\\
||{ {u}}^n-\widetilde{U}^n||_1&\leq C_2\sqrt{\frac{TL }{2\varepsilon_0}}(\tau^{\min\{3-\gamma_s,2-\alpha_q,2-\beta\}}+h^2),
\end{align*}
where $\varepsilon_0=\sum\limits_{j=1}^{s}\frac{b_j T^{1-\gamma_j}}{2\Gamma(2-\gamma_j)}+a_1$.
\end{thm}
\begin{proof}
Denote $e_i^n=u_{i}^{n}-U_i^n$, $e^n=[e_1^n,e_2^n,\ldots,e_{M-1}^n]^T$. Subtracting (\ref{e22}) from (\ref{e21}), we have
\begin{align*}
&\sum_{j=1}^{s}b_j\mu_{1,j}\Big[a_0^{(\gamma_j)}\nabla_te_{i}^{n}-
\sum_{k=1}^{n-1}(a_{n-k-1}^{(\gamma_j)}-a_{n-k}^{(\gamma_j)})\nabla_te_{i}^{k}
\Big]\nonumber\\
+&a_1\nabla_te_{i}^{n}+\sum_{l=1}^{q}c_l\mu_{2,l}
\sum_{k=1}^{n}d_{n-k}^{(\alpha_l)} \nabla_te_i^{k}
+a_2e_{i}^{n}\nonumber\\
=&a_3\delta_x^2e_{i}^{n}+a_4\mu_3
\sum_{k=1}^{n}d_{n-k}^{(\beta)} \nabla_t(\delta_x^2e_i^{k})+R_{1,i}^n,
\end{align*}
with $e_i^0=0$, $e_0^n=e_M^n=0$.
Then from (\ref{e33}) and (\ref{e34}), we have
 \begin{align*}
{a_2}||e^{n}||_0^2+{a_3}|e^{n}|_1^2\leq \frac{\tau h}{2\varepsilon_0}
\sum_{k=1}^{n}\sum_{i=1}^{M-1}(R_{1,i}^n)^2\leq \frac{\tau h}{2\varepsilon_0}
\sum_{k=1}^{n}\sum_{i=1}^{M-1}C_1^2(\tau+h^2)^2
\leq \frac{C_1^2TL }{2\varepsilon_0}
(\tau+h^2)^2,
\end{align*}
namely,
\begin{align*}
||{u}^n-U^n||_1^2&\leq \frac{C_1^2TL }{2\varepsilon_0}(\tau+h^2)^2.
\end{align*}
Similarly, we can obtain
\begin{align*}
||{{u}}^n-\widetilde{U}^n||_1^2&\leq \frac{C_2^2TL }{2\varepsilon_0}(\tau^{\min\{3-\gamma_s,2-\alpha_q,2-\beta\}}+h^2)^2.
\end{align*}
\end{proof}

\section{Numerical examples}

\textbf{Example 1}\quad We consider the following multi-term time fractional viscoelastic non-Newtonian fluid model.
\begin{eqnarray*}
   \begin{cases}
   ~{D_t^\gamma} u(x,t)+\frac{\partial u(x,t)}{\partial {t}}+{D_t^\alpha} u(x,t)+u(x,t)=\frac{\partial^2 u(x,t)}{\partial {x^2}}+{D_t^\beta}\frac{\partial^2 u(x,t)}{\partial {x^2}}+f(x,t),~(x,t)\in(0,1)\times(0,1], \\[+2pt]
   ~u(x,0)=\sin\pi x,\quad u_t(x,0)=0,\quad 0\le x\le 1,\\[+1pt]
   ~u(0,t)=0,\quad u(1,t)=0,\quad\quad 0\le t\le 1,
\end{cases}
\end{eqnarray*}
where $0<\alpha, \beta<1$, $1<\gamma<2$,
\begin{eqnarray*}
f(x,t)=\sin\pi x\Big[\frac{\Gamma(4)t^{3-\gamma}}{\Gamma(4-\gamma)}+3t^2
+\frac{\Gamma(4)t^{3-\alpha}}{\Gamma(4-\alpha)}+(1+\pi^2)(t^3+1)+
\frac{\pi^2\Gamma(4)t^{3-\beta}}{\Gamma(4-\beta)}\Big]
\end{eqnarray*}
and the exact solution is $u(x,t)=(t^3+1)\sin\pi x$.

Firstly, we use the implicit finite difference scheme (\ref{e22}) (Scheme I) to solve the equation and the numerical results are given in Table 1. The table lists the $L_2$ error and $L_{\infty}$ error and the convergence order of $\tau$ for different $\alpha$, $\beta$, $\gamma$ with $h=1/1000$ at $t=1$. We can see that the numerical results are in perfect agreement with the exact solution and the convergence order reaches the expected first order. Then we apply the mixed L scheme (\ref{e25}) (Scheme II) to the equation. Table 2 displays the $L_2$ error and $L_{\infty}$ error and convergence order of $\tau$ for different $\alpha$, $\beta$, $\gamma$ with $h=1/1000$ at $t=1$. We can observe that the numerical results are in excellent agreement with the exact solution and the convergence order attains the expected $\min\{3-\gamma,2-\alpha,2-\beta\}$ order.  Compared to scheme I, the results of scheme II are more accurate. In addition, we present a comparison of CPU time for two schemes in Table 3. Here the numerical computations were carried out using MATLAB R2014b on a Dell desktop with configuration: Intel(R) Core(TM) i7-4790, 3.60 GHz and 16.0 GB RAM. We choose $\alpha=0.7,~\beta=0.6,~\gamma=1.5$ and $h=1/1000$ at $t=1$ to observe the running time for different $\tau$. We observe that the running time of Scheme II is more than that of scheme I, which dues to the fact that more terms are added in Scheme II. 
 \begin{table}[htbp]
\caption{The temporal error and convergence of Scheme I for different $\alpha$,
$\beta$ and $\gamma$ with $h=1/1000$.} \centering
\begin{tabular}{c  c c c c}
 \toprule
$\alpha=0.7,~\beta=0.6,~\gamma=1.5$      & $ ||E(h,\tau)||_0$ &  Order  &  $ ||E(h,\tau)||_{\infty}$  &  Order    \\
\midrule
$1/40$  & 7.0478E-03     &        &  9.9671E-03      &          \\
$1/80$  & 3.2211E-03     &  1.13  &  4.5553E-03      &  1.13    \\
$1/160$ & 1.4899E-03     &  1.11  &  2.1071E-03      &  1.11    \\
$1/320$ & 6.9792E-04     &  1.09  &  9.8700E-04      &  1.09    \\
$1/640$ & 3.3092E-04     &  1.08  &  4.6799E-04      &  1.08    \\
\midrule
$\alpha=0.7,~\beta=0.8,~\gamma=1.6$         & $ ||E(h,\tau)||_0$ &  Order  &  $ ||E(h,\tau)||_{\infty}$  &  Order   \\
\midrule
$1/40$  & 1.0895E-02     &        &  1.5408E-02      &          \\
$1/80$  & 5.0166E-03     &  1.12  &  7.0946E-03      &  1.12    \\
$1/160$ & 2.3164E-03     &  1.11  &  3.2759E-03      &  1.11    \\
$1/320$ & 1.0742E-03     &  1.11  &  1.5191E-03      &  1.11    \\
$1/640$ & 5.0073E-04     &  1.10  &  7.0814E-04      &  1.10    \\
 \midrule
$\alpha=0.5,~\beta=0.3,~\gamma=1.6$ & $ ||E(h,\tau)||_0$ &  Order  &  $ ||E(h,\tau)||_{\infty}$  &  Order    \\
\midrule
$1/40$  & 5.5522E-03     &        &  7.8520E-03      &          \\
$1/80$  & 2.6575E-03     &  1.06  &  3.7583E-03      &  1.06    \\
$1/160$ & 1.2862E-03     &  1.05  &  1.8190E-03      &  1.05    \\
$1/320$ & 6.2820E-04     &  1.03  &  8.8841E-04      &  1.03    \\
$1/640$ & 3.0906E-04     &  1.02  &  4.3707E-04      &  1.02    \\
\bottomrule
\end{tabular}
\end{table}

\begin{table}[htbp]
\caption{The temporal error and convergence of Scheme II for different $\alpha$,
$\beta$ and $\gamma$ with $h=1/1000$.} \centering
\begin{tabular}{c  c c c c}
 \toprule
$\alpha=0.7,~\beta=0.6,~\gamma=1.5$      & $ ||E(h,\tau)||_0$ &  Order  &  $ ||E(h,\tau)||_{\infty}$  &  Order    \\
\midrule
$1/40$  & 2.8002E-03     &        &  3.9601E-03      &          \\
$1/80$  & 1.0828E-03     &  1.37  &  1.5313E-03      &  1.37    \\
$1/160$ & 4.1712E-04     &  1.38  &  5.8989E-04      &  1.38    \\
$1/320$ & 1.6057E-04     &  1.38  &  2.2708E-04      &  1.38    \\
$1/640$ & 6.2009E-05     &  1.37  &  8.7694E-05      &  1.37    \\
\midrule
$\alpha=0.7,~\beta=0.8,~\gamma=1.6$         & $ ||E(h,\tau)||_0$ &  Order  &  $ ||E(h,\tau)||_{\infty}$  &  Order   \\
\midrule
$1/40$  & 6.7356E-03     &        &  9.5256E-03      &          \\
$1/80$  & 2.9253E-03     &  1.20  &  4.1370E-03      &  1.20    \\
$1/160$ & 1.2677E-03     &  1.21  &  1.7928E-03      &  1.21    \\
$1/320$ & 5.4905E-04     &  1.21  &  7.7648E-04      &  1.21    \\
$1/640$ & 2.3795E-04     &  1.21  &  3.3652E-04      &  1.21    \\
 \midrule
$\alpha=0.5,~\beta=0.3,~\gamma=1.6$ & $ ||E(h,\tau)||_0$ &  Order  &  $ ||E(h,\tau)||_{\infty}$  &  Order    \\
\midrule
$1/40$  & 8.8331E-04     &        &  1.2492E-03      &          \\
$1/80$  & 3.0948E-04     &  1.51  &  4.3767E-04      &  1.51    \\
$1/160$ & 1.0881E-04     &  1.51  &  1.5388E-04      &  1.51    \\
$1/320$ & 3.8658E-05     &  1.49  &  5.4671E-05      &  1.49    \\
$1/640$ & 1.4074E-05     &  1.46  &  1.9904E-05      &  1.46    \\
\bottomrule
\end{tabular}
\end{table}

\begin{table}[htbp]
\caption{The running time of Scheme I and II for different $\tau$,
with $\alpha=0.7,~\beta=0.6,~\gamma=1.5$, $h=1/1000$ at $t=1$.} \centering
\begin{tabular}{c  c c }
 \toprule
$\tau$    &  Scheme I &  Scheme II     \\
\midrule
$1/40$    & 2.9279s     &  3.5410s         \\
$1/80$    & 9.5444s     &  10.6752s        \\
$1/160$   & 33.2557s    &  35.8488s        \\
$1/320$   & 123.4362s   &  130.4436s       \\
$1/640$   & 475.8209s   &  492.6424s       \\
\bottomrule
\end{tabular}
\end{table}


\textbf{Example 2}\quad Next, we consider the following unsteady MHD Couette flow of a generalized Oldroyd-B fluid \cite{a12}.
\begin{eqnarray*}
   \begin{cases}
   ~(1+\lambda^\alpha{D}_t^{\alpha})\frac{\partial u(x,t)}{\partial t}=(1+\theta^\beta{D}_t^\beta)\frac{\partial^2u(x,t)}{\partial x^2}-K(1+\lambda^\alpha{D}_t^{\alpha})u(x,t),~(x,t)\in(0,1)\times(0,2] \\[+1pt]
   ~u(x,0)=0,\quad u_t(x,0)=0,\quad 0\le x\le 1,\\
   ~u(0,t)=0,\quad u(1,t)=2t^p,\quad\quad 0\le t\le 2,
\end{cases}
\end{eqnarray*}
where $0<\alpha,\beta<1$, $\lambda$ is the relaxation time, $\theta$ is the retardation time, $K=\frac{\sigma B_0^2}{\rho}$, $\rho$ is the density of the fluid, $B_0$ is the magnetic intensity and $\sigma$ is the electrical conductivity, and $\lambda,\theta,K\geq 0$.

This model describes the flow of an incompressible Olyroyd-B fluid bounded by two infinite parallel rigid plates in a magnetic field. Initially, the whole system is at rest and the lower plate is fixed. Then at time $t=0^+$, the upper plate starts to slide with some velocity $2t^p$ in the main flow direction, which is termed as the plane Couette flow. Due to the influence of shear, the fluid is gradually in motion.

In the calculations, we choose $h=1/1000$, $\tau=1/100$. In order to observe the effects of different physical parameters on the velocity field, we plot some figures to demonstrate the dynamic characteristics of the generalized Oldroyd-B fluid. Fig. 2 shows the effects of power-law index $p$ and constant $K$ on the velocity. We can see that the velocity increases with increasing power-law index $p$ and the magnetic body force is favorable to the decays of the velocity. We can clearly find that the greater $K$ is, the more rapidly the velocity decays. Fig. 3 exhibits the effects of the relaxation time $\lambda$ and the retardation time $\theta$ on the velocity. We can see that the smaller the $\lambda$, the more slowly the velocity decays. However, an opposite trend for the variation of $\theta$ can be seen. Fig. 4 illustrates the change of the velocity with different parameters $\alpha$ and $\beta$. It is seen that the larger the value of $\alpha$, the more rapidly the velocity decays. The effect of $\beta$ is contrary to that of $\alpha$. Fig. 5 depicts the influence of time on the velocity and we can observe that the flow velocity increases with increasing time.
\begin{figure}[htbp]
\centering
\scalebox{0.4}[0.4]{\includegraphics{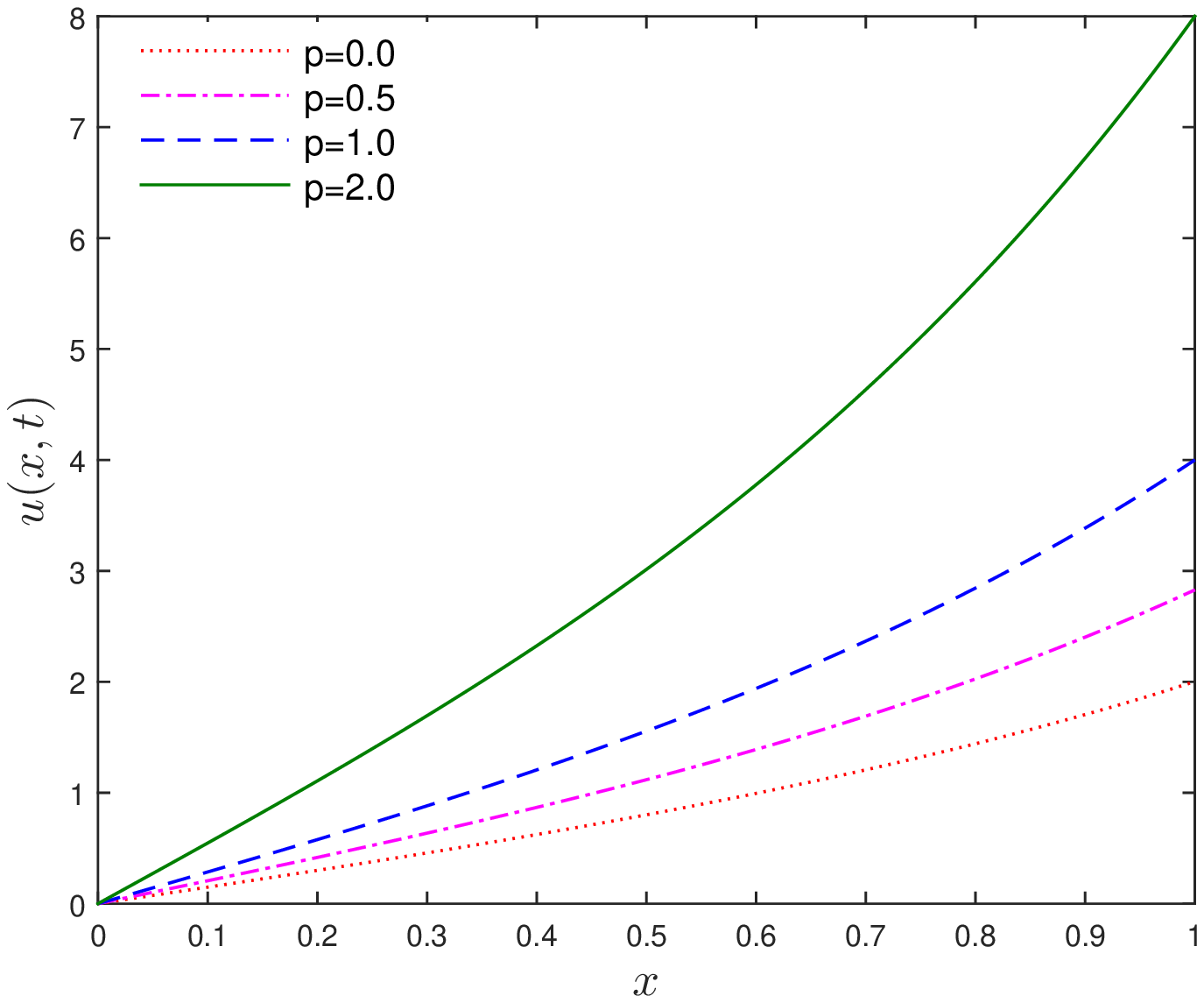}}
\scalebox{0.4}[0.4]{\includegraphics{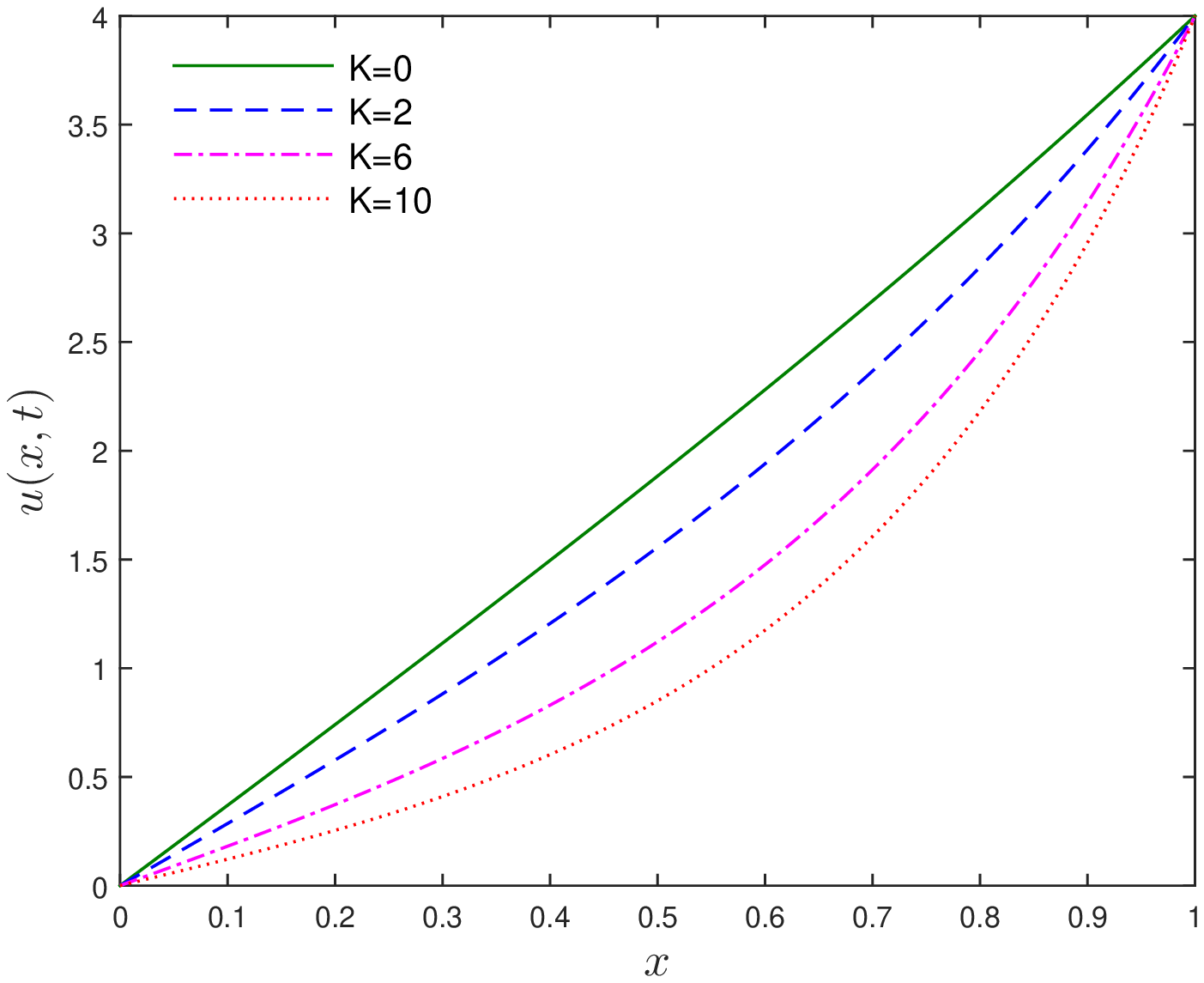}}
\caption{Numerical solution profiles of velocity $u(x,t)$  for different $p$ ($K=2$) and $K$ ($p=1$) with $\lambda=3$, $\theta=4$, $\alpha=0.5$, $\beta=0.6$ at $t=2$.}
\end{figure}
\begin{figure}[htbp]
\centering
\scalebox{0.4}[0.4]{\includegraphics{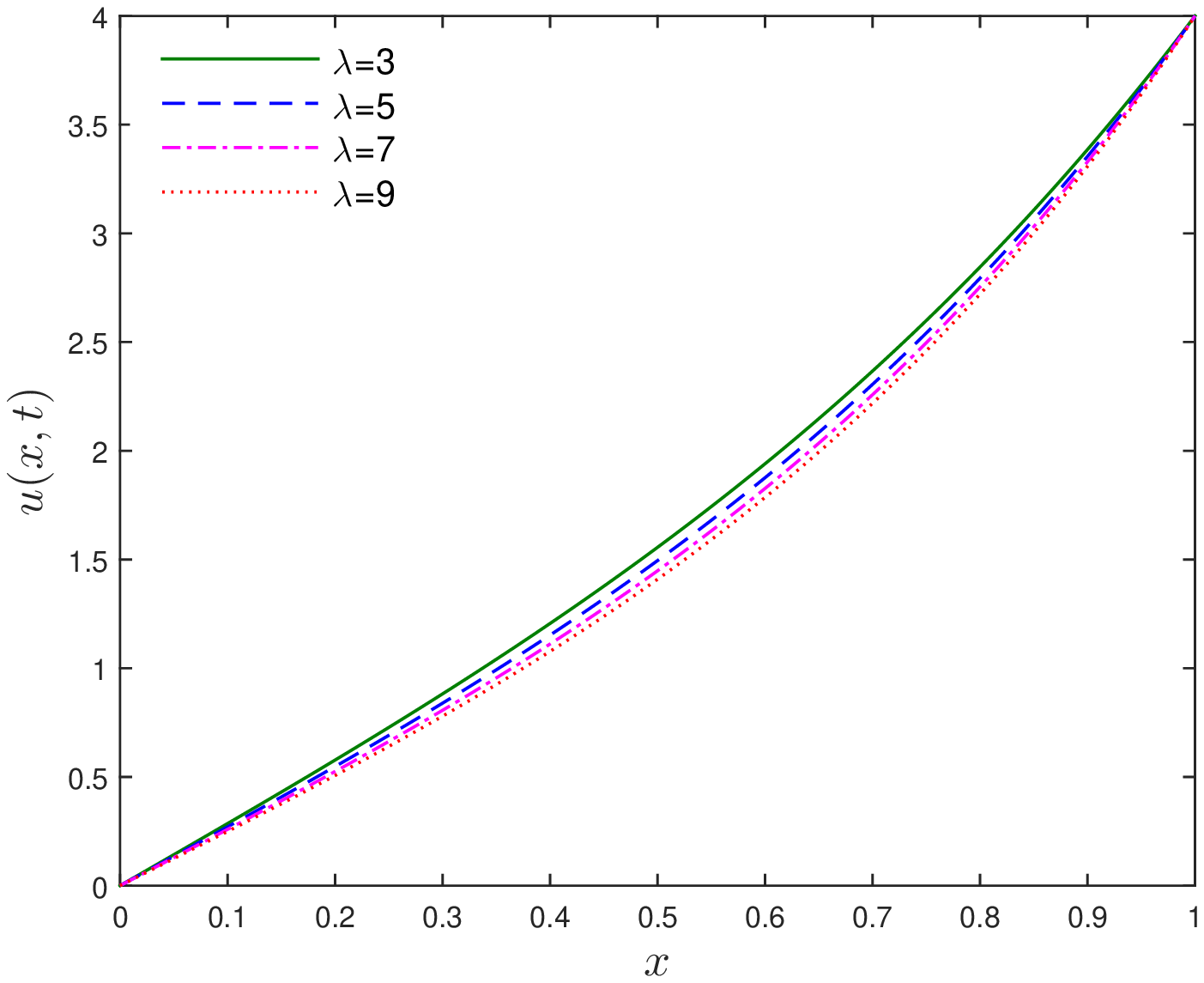}}
\scalebox{0.4}[0.4]{\includegraphics{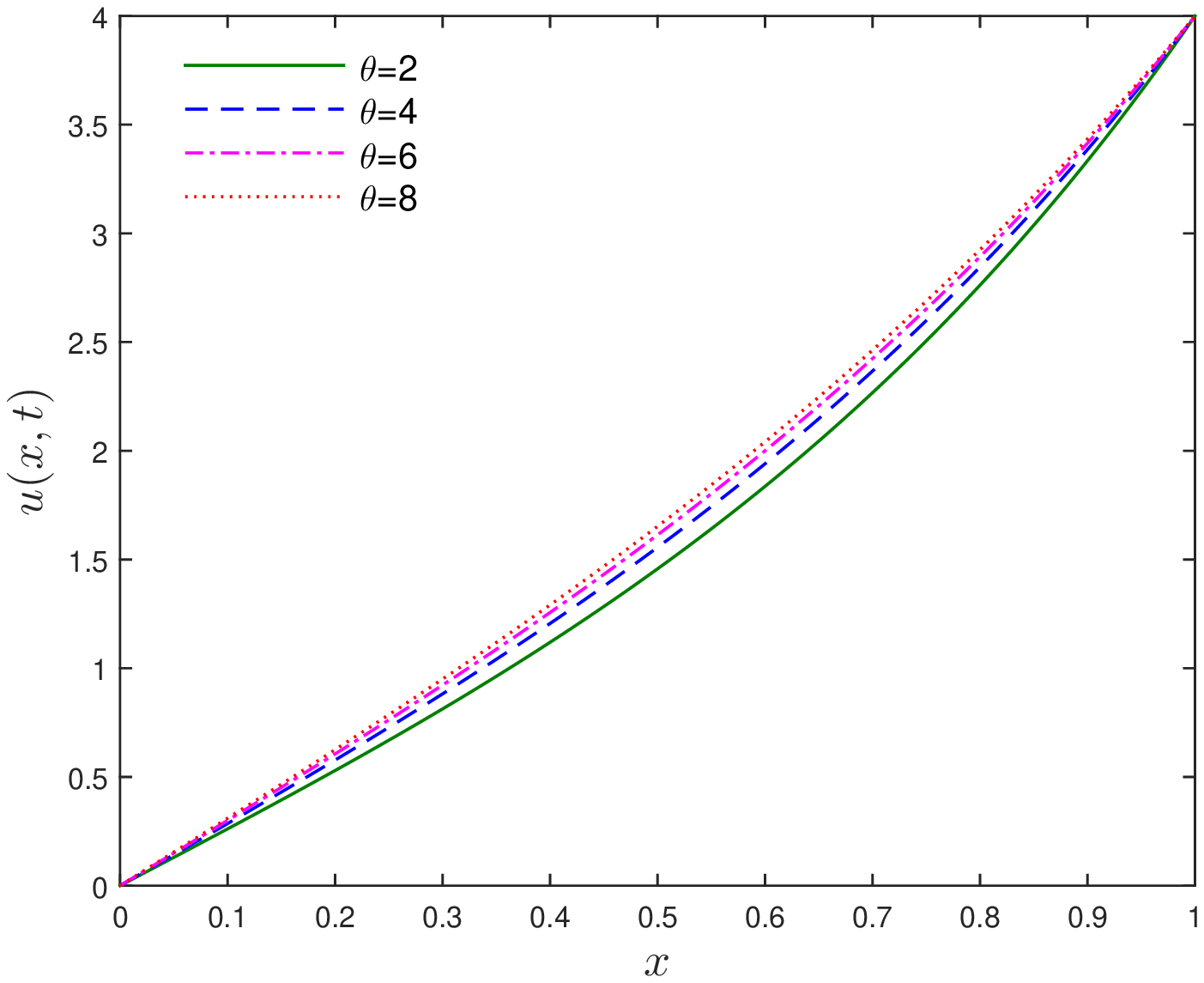}}
\caption{Numerical solution profiles of velocity $u(x,t)$  for different $\lambda$ ($\theta=4$) and $\theta$ ($\lambda=3$) with $p=1$, $\alpha=0.5$, $\beta=0.6$, $K=2$ at $t=2$.}
\end{figure}
\begin{figure}[htbp]
\centering
\scalebox{0.4}[0.4]{\includegraphics{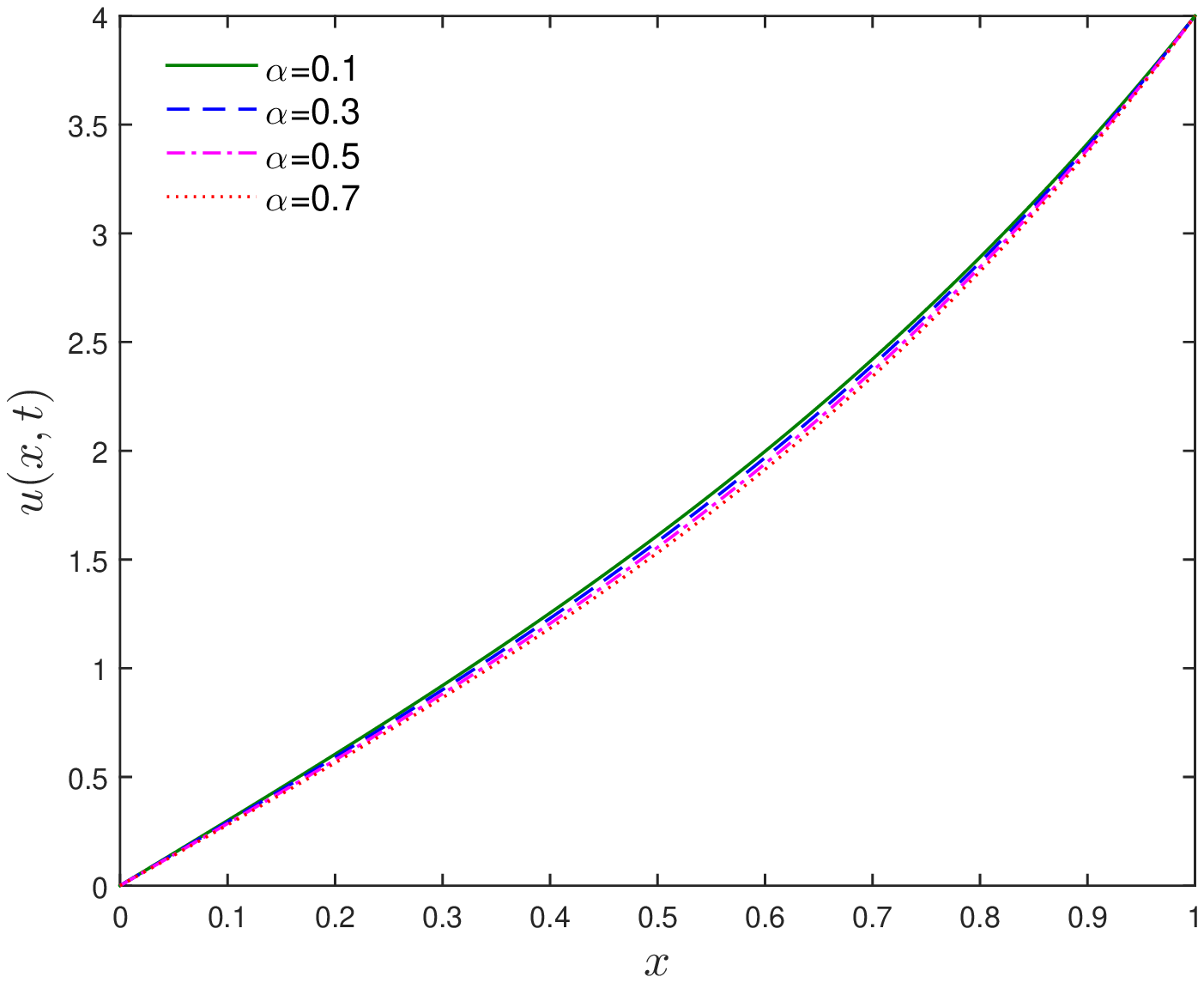}}
\scalebox{0.4}[0.4]{\includegraphics{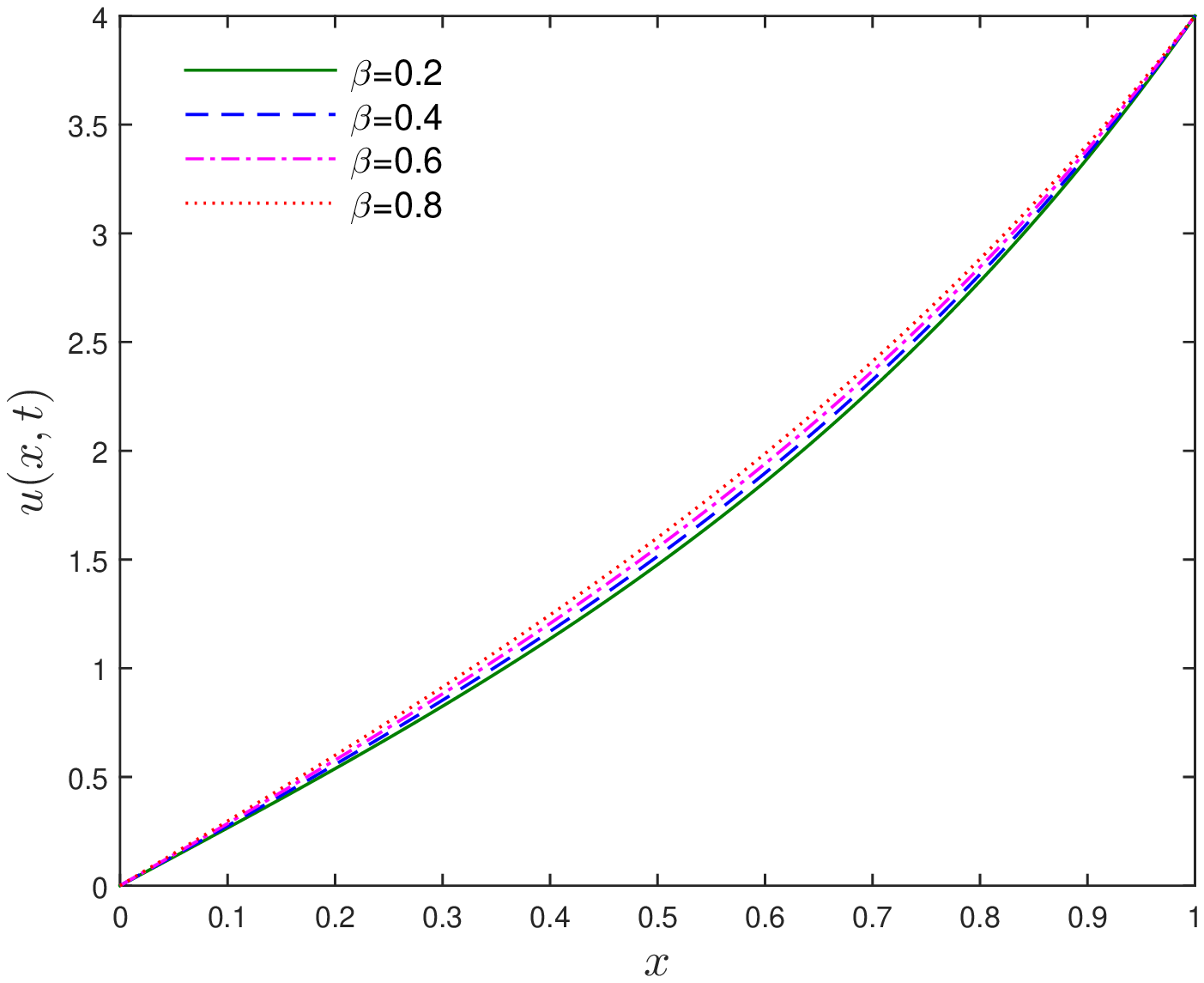}}
\caption{Numerical solution profiles of velocity $u(x,t)$  for different $\alpha$ ($\beta=0.6$) and  $\beta$ ($\alpha=0.5$) with $p=1$, $\lambda=3$, $\theta=4$, $K=2$ at $t=2$.}
\end{figure}
\begin{figure}[htbp]
\centering
\scalebox{0.4}[0.4]{\includegraphics{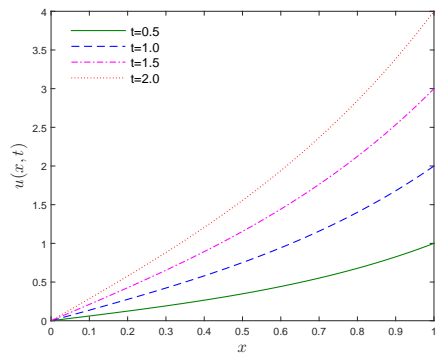}}
\caption{Numerical solution profiles of velocity $u(x,t)$  at different $t$ with $p=1$, $\lambda=3$, $\theta=4$, $\alpha=0.5$, $\beta=0.6$, $K=2$.}
\end{figure}

\section{Conclusions}

In this paper, we proposed the finite difference method to solve the novel multi-term time fractional viscoelastic non-Newtonian fluid model. We not only presented an implicit difference scheme with accuracy of $O(\tau+h^2)$, but also give a high order time scheme with accuracy of $O(\tau^{\min\{3-\gamma_s,2-\alpha_q,2-\beta\}}+h^2)$. In addition, we established the stability and convergence analysis of the finite difference schemes. Two numerical examples were exhibited to verify the effectiveness and reliability of our method. We can conclude that our numerical methods are robust and can be extended to other multi-term time fractional diffusion equations, such as the generalized Oldroyd-B fluid in a rotating system, the generalized Maxwell fluid model, the generalized second grade fluid model and the generalized Burgers' fluid model. In future work, we shall investigate the application of the these methods and techniques to the novel multi-term time fractional viscoelastic non-Newtonian fluid model in high dimensional cases.

\end{document}